\documentclass[12pt]{amsart}

\usepackage{fullpage}
\usepackage{tikz}
\usepackage{hyperref}
\usepackage[all,cmtip]{xy}
\usepackage{amsopn}
\usepackage{amsfonts}
\usepackage{amsmath}
\usepackage{amsthm}
\usepackage{amscd}
\usepackage{amssymb}
\usepackage{graphicx}
\usepackage{times} 

\makeatletter\@addtoreset {equation}{section}\makeatother

\theoremstyle{plain}
\newtheorem{definition}{Definition}[section]

\newtheorem{theorem}[definition]{Theorem}
\newtheorem{corollary}[definition]{Corollary}
\newtheorem{lemma}[definition]{Lemma}

\newtheorem{remark}[definition]{Remark}

\newtheorem{conjecture}[definition]{Conjecture}
\DeclareMathOperator{\sech}{sech}

\newcommand*\diff{\mathop{}\!\mathrm{d}}

\newenvironment{proof1}%
{\begin{trivlist} \item[]{\em Proof }}%
{\hspace*{\fill}$\rule{.3\baselineskip}{.35\baselineskip}$\end{trivlist}}

\begin{document}

\title[Spectral Stability of Shifted States on Star Graphs]
{Spectral Stability of Shifted States on Star Graphs}

\author{Adilbek Kairzhan}
\address{Department of Mathematics, McMaster University, Hamilton, Ontario  L8S 4K1, Canada}
\email{kairzhaa@math.mcmaster.ca}

\author{Dmitry E. Pelinovsky}
\address{Department of Mathematics, McMaster University, Hamilton, Ontario  L8S 4K1, Canada}
\email{dmpeli@math.mcmaster.ca}

\thanks{This research was supported by the NSERC Discovery Grant.}
\date{\today}

\begin{abstract}
We consider the nonlinear Schr\"{o}dinger (NLS) equation with the subcritical power nonlinearity on a star graph
consisting of $N$ edges and a single vertex under generalized Kirchhoff boundary conditions.
The stationary NLS equation may admit a family of solitary waves parameterized by a translational parameter,
which we call the shifted states. The two main examples include
(i) the star graph with even $N$ under the classical Kirchhoff boundary conditions
and (ii) the star graph with one incoming edge and $N-1$ outgoing edges under a single constraint
on coefficients of the generalized Kirchhoff boundary conditions.
We obtain the general counting results on the Morse index of the shifted states and apply them
to the two examples. In the case of (i), we prove that the shifted states with even $N \geq 4$
are saddle points of the action functional which are spectrally unstable under the NLS flow.
In the case of (ii), we prove that the shifted states with the monotone profiles in the $N-1$ outgoing edges
are spectrally stable, whereas the shifted states with non-monotone profiles in the $N-1$ outgoing edges
are spectrally unstable, the two families intersect at the half-soliton states
which are spectrally stable but nonlinearly unstable under the NLS flow.
Since the NLS equation on a star graph with shifted states can be reduced to
the homogeneous NLS equation on an infinite line, the spectral instability of shifted states
is due to the perturbations breaking this reduction. We give a simple argument
suggesting that the spectrally stable shifted states are unstable
under the NLS flow due to the perturbations breaking the reduction to the NLS equation on 
an infinite line.
\end{abstract}

\maketitle


\section{Introduction}

It is well-known that multi-dimensional models can be reduced approximately to the PDEs on
metric graphs in many realistic physical experiments involving wave propagation in narrow waveguides
\cite{Beck,Joly1,Joly2,Kuch}. When a multi-dimensional narrow waveguide is replaced
by a graph consisting of edges and vertices, a metric space is introduced
for the wave functions on edges and is equipped with appropriate boundary conditions at vertices.
{\em Kirchhoff} boundary conditions appear naturally in many applications and require
the wave functions to be continuous across the
vertex points and the sum of the first derivatives of the wave functions to be zero,
see, e.g., books \cite{Kuchment,Exner}.

It is relatively less known that the Kirchhoff boundary conditions are not the only possible
boundary conditions arising when the narrow waveguides shrinks to a metric graph. By working with different
values of the thickness parameters vanishing in the limit at the same rate, it was shown in
\cite{Post} (see also \cite{Costa,EP2009,EP2013,Kosugi,Molchanov}) that
{\em generalized Kirchhoff} boundary conditions can also arise in the limit.
In the generalized Kirchhoff boundary conditions, the wave functions have finite jumps
across the vertex points and these jumps
are compensated reciprocally in the sum of the first derivatives of the wave function.
The Laplacian operator on the metric graphs with the generalized Kirchhoff boundary conditions
is still extended to a self-adjoint operator similarly to the case with the classical
Kirchhoff boundary conditions. Numerical confirmations of validity of the classical
and generalized Kirchhoff boundary conditions
are reported in a number of recent publications in physics literature \cite{Caputo,Sobirov1,Sobirov2}.

In a series of papers \cite{M1,M2,M3}, it was shown that if the parameters of the
generalized Kirchhoff boundary conditions on a star graph are related to the parameters
of the nonlinear evolution equations and satisfy a single constraint, then the nonlinear evolution equation
on the star graph can be reduced to the homogeneous equation on the infinite line. In other words,
singularities of the star graph are unfolded in the transformation and the vertex points become regular points on the line.
In this case, a transmission of a solitary wave through the vertex points will be reflectionless.

The novelty of this paper is to explore the star graphs with the generalized Kirchhoff boundary
conditions satisfying the constraint enabling reflectionless transmission of the solitary wave.
We focus our study on the case example of the nonlinear Schr\"{o}dinger (NLS) equation,
which is the simplest nonlinear evolution equation considered on metric graphs \cite{Noja}.
In the simplest configuration of a star graph with a single vertex, we study existence and
stability of stationary states of the NLS equation, which have a continuous parameter of
their translations along the graph. We call such stationary states as {\em the shifted states}.

The shifted states of the NLS equation on a star graph appear naturally in the case of classical Kirchhoff boundary
conditions when the number of edges is even. These states can be considered to be translations
of {\em the half-soliton states}, which exist for any number of edges and the nonlinear instability of which
was considered in our previous work \cite{KP} (see \cite{AdamiJPA} for the very first prediction of nonlinear
instability of the half-soliton states). In the variational characterization of the NLS stationary states
on a star graph, such shifted states were mentioned in Remarks 5.3 and 5.4 in \cite{AdamiJDE1}, where it was conjectured
that all shifted states are saddle points of the action functional and are thus unstable for all star graphs
with even number of edges exceeding two.

The purpose of this paper is to prove this conjecture with an explicit count of the Morse index for the shifted states.
By extending the Sturm theory to Schr\"{o}dinger operators on the star graph, we can give a
very precise characterization of the negative and zero eigenvalues of the linearized Schr\"{o}dinger
operators, avoiding the less explicit theory of deficiency index for star graphs with point interactions \cite{Pava}.
As a result of our analysis, we prove that these shifted states are saddle
points of energy subject to fixed mass, which are spectrally unstable under the NLS flow.
In comparison, the half-soliton states are degenerate saddle points of energy and they
are spectrally stable but nonlinearly unstable under the NLS flow.

We treat the star graph with an even number of edges as a particular example of the star graph
with the generalized Kirchhoff boundary conditions satisfying the constraint enabling
reflectionless transmission of the solitary wave \cite{M3}. In this more general context, we show
that the shifted states satisfy the reduction of the NLS equation on the star graph to
the homogeneous NLS equation on the infinite line. Nevertheless, we show that with one exception, the shifted states
are spectrally unstable in the time evolution due to perturbations that break this reduction.
Since numerical simulations have been performed in \cite{M1,M2,M3} with initial conditions satisfying
this reduction and with solitary waves transmitted across the vertex points, no instability of
the shifted states have been reported in the previous publications, to the best of our knowledge.

The only exception when the shifted states may be spectrally stable is the star graph with one
incoming edge and $N-1$ outgoing edges. In this case, we prove that the shifted states with
the monotone profiles in the $N-1$ outgoing edges are spectrally stable,
whereas the shifted states with non-monotone profiles in the $N-1$ outgoing edges
are spectrally unstable, the two families intersect at the nonlinearly unstable half-soliton state.
In spite of the spectral stability of the shifted states with the monotone profiles, we give a simple
argument that the shifted states are nonlinearly unstable
due to perturbations that break homogenization of the NLS equation.

Our conjecture is that the central peak of the shifted state
moves due to symmetry-breaking perturbations from the only incoming edge
towards the vertex, breaks into $(N-1)$ peaks in the outgoing edges, the latter peaks become unstable
due to spectral instability of the shifted states with the nonmonotone profiles.
This conjecture is based on the comparison of the energy of the central peak
in the incoming and outgoing edges at fixed mass. Due to this variational interpretation,
spectral and nonlinear instabilities of the shifted states are likely
to lead to a formation of a solitary wave escaping to infinity along one of the $(N-1)$ outgoing edges of the star graph.
This dynamical picture appears to be in a complete agreement with the nonexistence results in \cite{AdamiJFA}
developed for the ground state of energy at fixed mass in the context of the star graphs
with the classical Kirchhoff boundary conditions.

The paper is organized as follows. Section 2 gives a mathematical setup of the NLS equation on the star graph
under the generalized Kirchhoff boundary conditions. Section 3
introduces the family of shifted states and discusses their properties. Main results on spectral stability
of the shifted states are formulated in Section 4. The main results are proven in Section 5 with the count
of the Morse index for the shifted states. Section 6 discusses homogenization of the star graph with the
NLS equation on the infinite line. Variational interpretation of our results
and our conjectures are described in Section 7.

\section{The NLS equation on the star graph}

Let $\Gamma$ be a star graph, which consists of $N$ half-lines connected at a common vertex.
The vertex is chosen as the origin and each edge of the star graph is parameterized
by $\mathbb{R}^+$. The Hilbert space on the graph $\Gamma$ is given by
$$
L^2(\Gamma) = \oplus_{j=1}^N L^2(\mathbb{R}^+).
$$
Elements in $L^2(\Gamma)$ are represented in the componentwise sense as vectors
$\Psi = (\psi_1, \psi_2, \dots, \psi_N)^T$
of $L^2(\mathbb{R}^+)$-functions with each component corresponding to one edge.
The inner product and squared norm of such $L^2(\Gamma)$-functions are given by
$$
\langle \Psi, \Phi \rangle_{L^2(\Gamma)} :=
\sum_{j=1}^N \int_{\mathbb{R}^+} \psi_j(x) \overline{\phi_j(x)} dx,
\quad \| \Psi \|^2_{L^2(\Gamma)} :=
\sum_{j=1}^N \| \psi_j\|^2_{L^2(\mathbb{R}^+)}.
$$
Similarly, we define the $L^2$-based Sobolev spaces on the graph $\Gamma$
$$
H^k(\Gamma) = \oplus_{j=1}^N H^k(\mathbb{R}^+), \quad k \in \mathbb{N}
$$
and equip them with suitable boundary conditions at the vertex.

For $k = 1$, we set generalized continuity boundary conditions as follows:
\begin{equation}
\label{H1}
H_\Gamma^1 := \{ \Psi \in H^1(\Gamma): \quad \alpha_1^{1/p} \psi_1(0) = \alpha_2^{1/p} \psi_2(0) = \dots = \alpha_N^{1/p} \psi_N(0)\},
\end{equation}
where $\alpha_1$, $\alpha_2$, $\dots$, $\alpha_N$ are positive coefficients.
These coefficients arise naturally when the one-dimensional star graph is obtained as a limit
of a narrow two-dimensional waveguide with different values of the thickness parameters
that go to zero at the same rate \cite{EP2009,EP2013,Post}.

For $k = 2$, we set generalized Kirchhoff boundary conditions as follows:
\begin{equation}
\label{H2}
H_\Gamma^2 := \left\{ \Psi \in H^2(\Gamma) \cap H^1_{\Gamma} : \quad
\sum_{j=1}^N  \frac{1}{\alpha_j^{1/p}} \psi_j'(0) = 0 \right\},
\end{equation}
where the prime stands for one-sided derivatives in $x$. The reason why the derivatives
depend reciprocally on the positive coefficients $\alpha_1$, $\alpha_2$, $\dots$,
$\alpha_N$ is due to the requirement on the existence of a self-adjoint extension
of the Laplacian operator in $L^2(\Gamma)$, as in the following lemma.

\begin{lemma}
\label{lemma-Laplacian}
There exists a self-adjoint extension of the Laplacian operator
$$
\Delta: H^2_\Gamma \subset L^2(\Gamma) \to L^2(\Gamma).
$$
\end{lemma}

\begin{proof}
If $\Psi \in H^2(\Gamma)$, then $\Psi(x), \Psi'(x) \to 0$ as $x \to \infty$ by Sobolev embedding theorem.
Therefore, for any $\Psi, \Phi \in H^2(\Gamma)$, integration by parts and boundary conditions
in (\ref{H1}) and (\ref{H2}) yield
\begin{eqnarray*}
\langle \Delta \Psi, \Phi \rangle_{L^2(\Gamma)} & = &
\langle \Psi, \Delta \Phi \rangle_{L^2(\Gamma)} + \sum_{j=1}^N \psi_j(0) \overline{\phi}'_j(0) - \sum_{j=1}^N \psi_j'(0) \overline{\phi}_j(0) \\
& = & \langle \Psi, \Delta \Phi \rangle_{L^2(\Gamma)}
+ \alpha_1^{1/p} \psi_1(0) \sum_{j=1}^N \alpha_j^{-1/p} \overline{\phi}'_j(0)
- \alpha_1^{1/p} \overline{\phi}_1(0) \sum_{j=1}^N \alpha_j^{-1/p} \psi_j'(0) \\
& = & \langle \Psi, \Delta \Phi \rangle_{L^2(\Gamma)}.
\end{eqnarray*}
By Theorem 1.4.4 in \cite{Kuchment}, the Laplacian operator $\Delta : L^2(\Gamma) \to L^2(\Gamma)$
with the domain $H^2_{\Gamma} \subset L^2(\Gamma)$ is extended to a self-adjoint operator.
\end{proof}

The nonlinear Schr\"{o}dinger (NLS) equation is posed on the star graph $\Gamma$
with the power nonlinearity:
\begin{equation} \label{eq1}
i \frac{\partial \Psi}{\partial t} = - \Delta \Psi - (p+1) \alpha^2 |\Psi|^{2p} \Psi, \quad x \in \Gamma, \quad t \in \mathbb{R},
\end{equation}
where $\Psi = \Psi(t,x) = (\psi_1,\psi_2,\dots,\psi_N)^T \in \mathbb{C}^N$, 
$\Delta : L^2(\Gamma) \to L^2(\Gamma)$ is the Laplacian operator
in Lemma \ref{lemma-Laplacian}, $\alpha \in L^{\infty}(\Gamma)$ is a piecewise constant function
with the coefficients $(\alpha_1, \alpha_2, \dots, \alpha_N) \in \mathbb{R}^N_+$ defined on the edges of $\Gamma$,
and the nonlinear term $\alpha^2 |\Psi|^{2p} \Psi$ is interpreted as a symbol for
$$
(\alpha_1^2|\psi_1|^{2p}\psi_1, \alpha_2^2 |\psi_2|^{2p}\psi_2, \dots, \alpha_N^2 |\psi_N|^{2p} \psi_N)^T.
$$
The constant coefficients $(\alpha_1, \alpha_2, \dots, \alpha_N)$ are the same
as in the boundary conditions (\ref{H1}) and (\ref{H2}).

The NLS equation (\ref{eq1}) is invariant under the gauge transformation $\Psi \mapsto e^{i\theta} \Psi$
and under the time translation $\Psi(t,x) \mapsto \Psi(t+t_0,x)$ with $\theta \in \mathbb{R}$ and $t_0 \in \mathbb{R}$.
The following lemma summarizes relevant results on the local-wellposedness of
the Cauchy problem and on the conservation of the mass and energy functionals.

\begin{lemma}
\label{lemma-wellposedness}
For every $p > 0$ and every $\Psi(0) \in H^1_{\Gamma}$, there exists $t_0 > 0$ and a local solution
\begin{equation}
\label{solution-in-H1}
\Psi(t) \in C((-t_0,t_0),H^1_{\Gamma}) \cap C^1((-t_0,t_0),H^{-1}(\Gamma))
\end{equation}
to the Cauchy problem associated with the NLS equation (\ref{eq1}) such that the mass
\begin{equation}
\label{mass}
Q(\Psi) := \| \Psi \|_{L^2(\Gamma)}^2,
\end{equation}
and the energy
\begin{equation}
\label{energy}
E(\Psi) := \| \Psi' \|_{L^2(\Gamma)}^2 - \| \alpha^{\frac{1}{p+1}} \Psi \|_{L^{2p+2}(\Gamma)}^{2p+2}
\end{equation}
are constant in $t \in (-t_0,t_0)$.
\end{lemma}

\begin{proof}
Local well-posedness of the NLS equation (\ref{eq1}) in $H^1_{\Gamma}$ is proved by using a standard contraction method
thanks to the isometry of the semi-group $e^{it \Delta}$ in $H^1_{\Gamma}$ and the Sobolev embedding of $H^1_{\Gamma}$ into $L^{\infty}(\Gamma)$.

Let us prove the mass and energy conservation under simplifying assumptions $p > 1/2$ and $p \geq 1$ respectively.
If $p > 1/2$ and $\Psi(0) \in H^2_{\Gamma}$, it follows from the contraction method that
there exists $t_0 > 0$ and a local strong solution
\begin{eqnarray}
\label{solution-in-H2}
\Psi(t) \in C((-t_0,t_0),H_\Gamma^2) \cap C^1((-t_0,t_0), L^2(\Gamma))
\end{eqnarray}
to the NLS equation (\ref{eq1}). Applying time derivative to $Q(\Psi)$ and using the NLS equation
(\ref{eq1}) yield the mass balance equation:
\begin{eqnarray*}
\frac{d}{dt} Q(\Psi) & = & -i \langle - \Delta \Psi - (p+1) \alpha^2 |\Psi|^{2p} \Psi, \Psi \rangle_{L^2(\Gamma)}
+ i \langle \Psi, - \Delta \Psi - (p+1) \alpha^2 |\Psi|^{2p} \Psi \rangle_{L^2(\Gamma)} \\
& = & i \langle \Delta \Psi, \Psi \rangle_{L^2(\Gamma)}
- i \langle \Psi, \Delta \Psi \rangle_{L^2(\Gamma)} = 0,
\end{eqnarray*}
where the last equality is obtained by Lemma \ref{lemma-Laplacian}.
Thus, the mass conservation of (\ref{mass}) is proven for $\Psi(0) \in H^2_{\Gamma}$.

If $p > 1/2$ and $\Psi(0) \in H^1_{\Gamma}$ but $\Psi(0) \notin H^2_{\Gamma}$,
then in order to prove the mass conservation of (\ref{mass}),
we define an approximating sequence $\{ \Psi^{(n)}(0) \}_{n \in \mathbb{N}}$ in $H^2_{\Gamma}$
such that $\Psi^{(n)}(0) \to \Psi(0)$ in $H^1_{\Gamma}$ as $n \to \infty$.
For each $\Psi^{(n)}(0) \in H^2_{\Gamma}$, there exists a local strong solution
$\Psi^{(n)}(t)$ given by (\ref{solution-in-H2}) for $t \in (-t_0^{(n)},t_0^{(n)})$. By Gronwall's inequality,
there exists a positive constant $K$ which only depends on the $H^1(\Gamma)$ norm
of the local solution $\Psi^{(n)}(t)$ such that
$$
\| \Psi^{(n) \prime \prime}(t) \|_{L^2(\Gamma)} \leq K \| \Psi^{(n) \prime \prime}(0) \|_{L^2(\Gamma)}, \quad
t \in (-t_0^{(n)},t_0^{(n)}),
$$
hence, the local existence time $t_0^{(n)}$ is determined by the $H^1(\Gamma)$ norm
of the initial data $\Psi^{(n)}(0)$. Due to the convergence $\Psi^{(n)}(0) \to \Psi(0)$ in $H^1_{\Gamma}$,
this implies that there is $t_0 > 0$ that depends on the $H^1(\Gamma)$
norm of $\Psi(0)$ such that $t_0^{(n)} \geq t_0$ for every $n \in \mathbb{N}$. Moreover,
$\Psi^{(n)}(t) \to \Psi(t)$ in $H^1_{\Gamma}$ as $n \to \infty$ for every $t \in (-t_0,t_0)$.
Since $Q(\Psi^{(n)}(t)) = Q(\Psi^{(n)}(0))$ for every $t \in (-t_0,t_0)$,
the limit $n \to \infty$ and the strong convergence in $H^1_{\Gamma}$ implies that
$Q(\Psi(t)) = Q(\Psi(0))$ for every $t \in (-t_0,t_0)$.

In order to prove the energy conservation, let us define the space $H^3_{\Gamma}$ compatible with the NLS flow:
\begin{equation}
\label{H3}
H_{\Gamma}^3 := \left\{ \Psi \in H^3(\Gamma) \cap H^2_\Gamma : \quad
\alpha_1^{1/p} \psi_1''(0) = \alpha_2^{1/p} \psi_2''(0) = \dots = \alpha_N^{1/p} \psi_N''(0) \right\}.
\end{equation}
If $p \geq 1$ and $\Psi(0) \in H^3_{\Gamma}$, it follows from the contraction method that
there exists $t_0 > 0$ and a local strong solution if $\Psi(0) \in H^3_{\Gamma}$
\begin{eqnarray}
\label{solution-in-H3}
\Psi(t) \in C((-t_0,t_0),H_\Gamma^3) \cap C^1((-t_0,t_0), H^1_{\Gamma})
\end{eqnarray}
to the NLS equation (\ref{eq1}). Applying time derivative to $E(\Psi)$ and using the NLS equation
(\ref{eq1}) yield the energy balance equation:
\begin{eqnarray*}
\frac{d}{dt} E(\Psi) & = &
i \langle \Psi''', \Psi' \rangle_{L^2(\Gamma)} - i \langle \Psi', \Psi''' \rangle_{L^2(\Gamma)}\\
& \phantom{t} & + i (p+1) \langle \alpha^2 (|\Psi|^{2p})' \Psi, \Psi' \rangle_{L^2(\Gamma)}
- i(p+1) \langle \Psi', \alpha^2 (|\Psi|^{2p})' \Psi \rangle_{L^2(\Gamma)}\\
& \phantom{t} &
+ i(p+1) \langle \Psi^{p+1}, \alpha^2 \Psi^p \Delta \Psi\rangle_{L^2(\Gamma)}
- i(p+1) \langle \alpha^2 \Psi^p \Delta \Psi, \Psi^{p+1} \rangle_{L^2(\Gamma)} \\
& = & i \sum_{j=1}^N \psi_j'(0) \left[ \overline{\psi}_j''(0) + (p+1) \alpha_j^2 |\psi_j(0)|^{2p}
\overline{\psi}_j(0)  \right] \\
& \phantom{t} & - i \sum_{j=1}^N \overline{\psi}_j'(0) \left[  \psi_j''(0) + (p+1) \alpha_j^2 |\psi_j(0)|^{2p} \psi_j(0) \right],
\end{eqnarray*}
where the decay of $\Psi(x)$, $\Psi'(x)$, and $\Psi''(x)$ to zero at infinity has been used
for the solution in $H^3_{\Gamma}$. Due to the boundary conditions in (\ref{H1}), (\ref{H2}), and (\ref{H3}),
we obtain $\frac{d}{dt} E(\Psi) = 0$, that is, the energy conservation of (\ref{energy})
is proven for $\Psi(0) \in H^3_{\Gamma}$. The proof for $p \geq 1$ and $\Psi(0) \in H^1_{\Gamma}$ but $\Psi(0) \notin H^3_{\Gamma}$
is achieved by using an approximating sequence similarly to the argument above.

Finally, the proof can be extended to the local solution (\ref{solution-in-H1}) for
all values of $p > 0$ by using other approximation techniques,
see, e.g., Theorems 3.3.1, 3.3.5, and 3.39 in \cite{Caz} or the proof of Proposition 2.2 in \cite{AdamiJDE1}
for NLS on $\Gamma$ with $\alpha = 1$.
\end{proof}

Global existence in the NLS flow only holds in the subcritical case $p \in (0,2)$.
In what follows, the scope of this work will be mainly developed in the subcritical case.

\begin{lemma}
\label{lemma-global}
For every $p\in (0,2)$, the local solution (\ref{solution-in-H1}) in Lemma \ref{lemma-wellposedness}
is extended globally with $t_0 \to \infty$.
\end{lemma}

\begin{proof}
This follows by the energy conservation and the Gagliardo-Nirenberg inequality
$$
\| \alpha^{\frac{1}{p+1}} \Psi \|_{L^{2p+2}(\Gamma)}^{2p+2} \leq C_{p,\alpha} \| \Psi' \|_{L^2(\Gamma)}^{p} \| \Psi \|_{L^2(\Gamma)}^{p+2},
$$
for every $\alpha \in L^{\infty}(\Gamma)$, $\Psi \in H^1_{\Gamma}$, $p > 0$, where the constant $C_{p,\alpha} > 0$ depends on $p$ and
$\alpha$ but does not depend on $\Psi$.
\end{proof}

In this work, we assume that the coefficients $(\alpha_1,\alpha_2,\dots,\alpha_N)$ satisfy the following constraint:
\begin{equation}
\label{alpha_const}
\sum_{j=1}^K \frac{1}{\alpha_j^{2/p}} = \sum_{j = K+1}^N \frac{1}{\alpha_j^{2/p}},
\end{equation}
where $K$ edges represent incoming bonds whereas the remaining $N-K$ edges
represent outgoing bonds, clearly $K \neq 0$ and $K \neq N$.
In particular, we consider two examples of the general star graph $\Gamma$:

\begin{enumerate}
\item[(i)] $\alpha_j=1$ for all $j$ and $N$ is even. The boundary conditions in $H_\Gamma^1$ give the continuity
at the vertex and the boundary conditions in $H_\Gamma^2$ are referred to as the classical Kirchhoff conditions.
Constraint (\ref{alpha_const}) is satisfied with $K = N/2$.

\item[(ii)] $K = 1$ and $N \geq 3$. The graph $\Gamma$ consists of one incoming edge that splits into $N-1$ outgoing edges.
The constraint (\ref{alpha_const}) gives the reflectionless boundary conditions for the transmission
of a solitary wave across the single vertex, as in \cite{M1,M2,M3}.
\end{enumerate}

\begin{remark}
\label{remark-NLS}
If $N = 2$, then the constraint (\ref{alpha_const}) is only satisfied if $K = 1$ and $\alpha_1 = \alpha_2$.
In this case, the NLS equation (\ref{nls}) on the graph $\Gamma$ is completely equivalent to the homogeneous
NLS equation on the infinite line $\mathbb{R}$.
\end{remark}

The stationary states of the NLS equation on the star graph $\Gamma$ under the constraint (\ref{alpha_const}) include
families parameterized by the translational parameter, as in the following section.

\section{Stationary states on the star graph}

Stationary states of the NLS are given by the solutions of the form
$$
\Psi(t,x) = e^{i\omega t} \Phi_{\omega}(x),
$$
where $(\omega,\Phi_{\omega}) \in \mathbb{R} \times H^2_{\Gamma}$ is a real-valued solution of the stationary NLS equation,
\begin{equation}\label{eq2}
-\Delta \Phi_{\omega} - (p+1) \alpha^2 |\Phi_{\omega}|^{2p} \Phi_{\omega} = - \omega \Phi_{\omega}.
\end{equation}

No solution $\Phi_{\omega} \in H^2_{\Gamma}$ to the stationary NLS equation (\ref{eq2}) exist for $\omega \leq 0$
because $\sigma(-\Delta) \geq 0$ in $L^2(\Gamma)$ and $\Phi_{\omega}(x), \Phi_{\omega}'(x) \to 0$ as $x \to \infty$
if $\Phi_{\omega} \in H^2_{\Gamma}$ by Sobolev's embedding theorems. For $\omega > 0$, the scaling transformation
$\Phi_{\omega}(x) = \omega^{\frac{1}{2p}} \Phi(z)$ with $z = \omega^{\frac{1}{2}} x$
is employed to scale $\omega$ to unity. The normalized profile $\Phi \in H^2_{\Gamma}$ is now a solution
of the stationary NLS equation
\begin{equation}\label{eq3}
-\Delta \Phi + \Phi - (p+1)\alpha^2 |\Phi|^{2p} \Phi = 0.
\end{equation}

For every $N$ and $\alpha$, the stationary NLS equation (\ref{eq3}) has a particular solution
\begin{equation}
\label{half-soliton}
\Phi(x) = \left[ \begin{array}{c} \alpha_1^{-1/p} \\ \alpha_2^{-1/p} \\ \vdots \\ \alpha_N^{-1/p} \end{array} \right] \phi(x),
\quad \mbox{\rm with} \; \phi(x) = \sech^{\frac{1}{p}}(px).
\end{equation}
This solution is labeled in the previous paper \cite{KP} as the {\em half-soliton state}.
In this work, we are interested in the families of solitary waves parameterized by a translational
parameter, which are labeled as the {\em shifted states}. Such families exist
if $(\alpha_1,\alpha_2,\dots,\alpha_N)$ satisfy the constraint (\ref{alpha_const}).

The following lemma gives the existence of a family of shifted states
under the constraint (\ref{alpha_const}).

\begin{lemma}
\label{solutions-II}
For every $p > 0$ and every $(\alpha_1,\alpha_2,\dots,\alpha_N)$ satisfying the constraint (\ref{alpha_const}), there exists
a one-parameter family of solutions to the stationary NLS equation (\ref{eq3})
with any $p > 0$ given by $\Phi = (\phi_1,\dots,\phi_N)^T$ with components
\begin{equation}
\label{soliton-shifted-II}
\phi_j(x) = \begin{cases} \alpha_j^{-1/p} \phi(x+a), & j=1,\dots,K \\
\alpha_j^{-1/p} \phi(x-a), & j = K+1,\dots,N,
\end{cases}
\end{equation}
where $\phi(x) = \sech^{\frac{1}{p}}(px)$ and $a\in \mathbb{R}$ is arbitrary.
\end{lemma}

\begin{proof}
A general solution to the stationary NLS equation (\ref{eq3}) decaying to zero at infinity
is given by $\Phi = (\phi_1,\dots,\phi_N)^T$ with components
$$
\phi_j(x) = \alpha_j^{-1/p} \phi(x+a_j), \quad 1 \leq j \leq N,
$$
where $(a_1,\dots,a_N) \in \mathbb{R}^N$ are arbitrary parameters.
The continuity condition in $H_\Gamma^2$ imply that $|a_1| = \dots = |a_N|$.
In other words, for every $j=1, \dots, N$, there exists $m_j \in \{0, 1\}$,
such that $a_j = (-1)^{m_j} |a|$ for some $a \in \mathbb{R}$.
The Kirchhoff condition in $H_\Gamma^2$ is equivalent to
\begin{equation}
\label{proof_cond}
\phi'(|a|) \sum_{j=1}^N \frac{(-1)^{m_j}}{\alpha_j^{2/p}} =0.
\end{equation}
If $a = 0$, the equation (\ref{proof_cond}) holds since $\phi'(0) = 0$ and this yields
the half-soliton state in the form (\ref{half-soliton}). If $a \neq 0$, then the equation (\ref{proof_cond}) holds
due to the constraint (\ref{alpha_const}) if
\begin{equation}
\label{m_j}
\mbox{\rm either} \;\; m_j = \begin{cases}
1 \quad \textrm{for} \quad 1\leq j \leq K \\
0 \quad \textrm{for} \quad K+1\leq j \leq N
\end{cases} \quad \mbox{\rm or} \quad
m_j = \begin{cases}
0 \quad \textrm{for} \quad 1\leq j \leq K \\
1 \quad \textrm{for} \quad K+1\leq j \leq N
\end{cases}
\end{equation}
In both cases, the shifted state appears in the form (\ref{soliton-shifted-II}) with either $a < 0$ or $a > 0$.
\end{proof}

\begin{remark}
\label{remark-a-0}
The half-soliton state (\ref{half-soliton}) corresponds to the shifted state of Lemma
\ref{solutions-II} with $a = 0$.
\end{remark}

\begin{remark}
Besides the two choices specified in the proof of Lemma \ref{solutions-II},
there might be other $N$-tuples $(m_1, m_2, \dots, m_N) \in \{0,1\}^N$
such that the bracket in (\ref{proof_cond}) becomes zero. Such $N$-tuples generate
new one-parameter family different from the one given by Lemma \ref{solutions-II}
under the same constraint (\ref{alpha_const}). For instance,
if $\alpha_j = 1$ for all $j$ and $K = N/2$, there exist $C_N$ different shifted states
given by Lemma \ref{solutions} below with $C_N$ computed in (\ref{C-N}).
\end{remark}

The following lemma gives a full classification of
families of shifted states in case (i) (see also Theorem 5 in \cite{AdamiJDE1}).

\begin{lemma}
\label{solutions}
For $\alpha = 1$ and for even $N$, there exists $C_N$ one-parameter
families of solutions to the stationary NLS equation (\ref{eq3}) with any $p > 0$, where
\begin{equation}
\label{C-N}
C_N = \frac{N!}{2 [(N/2)!]^2}.
\end{equation}
Each family is generated from the simplest state $\Phi = (\phi_1,\dots,\phi_N)^T$ with components
\begin{equation}
\label{soliton-shifted}
\phi_j(x) = \begin{cases} \phi(x+a), & j=1,\dots,\frac{N}{2} \\
\phi(x-a), & j=\frac{N}{2}+1,\dots,N
\end{cases}
\end{equation}
where $\phi(x) = \sech^{\frac{1}{p}}(px)$ and $a\in \mathbb{R}$ is arbitrary,
after rearrangements between $N/2$ edges with $+a$ shifts and $N/2$ edges with $-a$ shifts.
\end{lemma}

\begin{proof}
A general solution to the stationary NLS equation (\ref{eq3}) decaying to zero at infinity
is given by
$$
\Phi = (\phi(x+a_1),\dots,\phi(x+a_N))^T,
$$
where $\phi(x) = \sech^{\frac{1}{p}}(px)$, and $(a_1,\dots,a_N) \in \mathbb{R}^N$ are arbitrary parameters.
The continuity condition in $H_\Gamma^2$ imply that $|a_1| = \dots = |a_N|$.
The Kirchhoff condition in $H_\Gamma^2$ is equivalent to
$\phi(a) \sum_{j=1}^N \tanh(a_j) = 0$, which together with the continuity condition
implies that the set $(a_1,\dots,a_N)$ has exactly $\frac{N}{2}$ positive elements and
exactly $\frac{N}{2}$ negative elements.
\end{proof}

\begin{remark}
If $N = 2$, then $C_2 = 1$. The only branch of shifted states in Lemma \ref{solutions} corresponds to the NLS solitary wave
translated along an infinite line $\mathbb{R}$, see Remark \ref{remark-NLS}.
\end{remark}

\begin{remark}
If $N = 3$, then $C_3 = 3$. The three branches of shifted states in Lemma \ref{solutions}
correspond to the three possible NLS solitary waves translated along an infinite line $\mathbb{R}$,
which is defined by the union of either $(1,2)$ or $(1,3)$, or $(2,3)$ edges of the star graph $\Gamma$.
\end{remark}

\begin{figure}[htb]
\begin{center}
\includegraphics[height=5cm,width=7.cm]{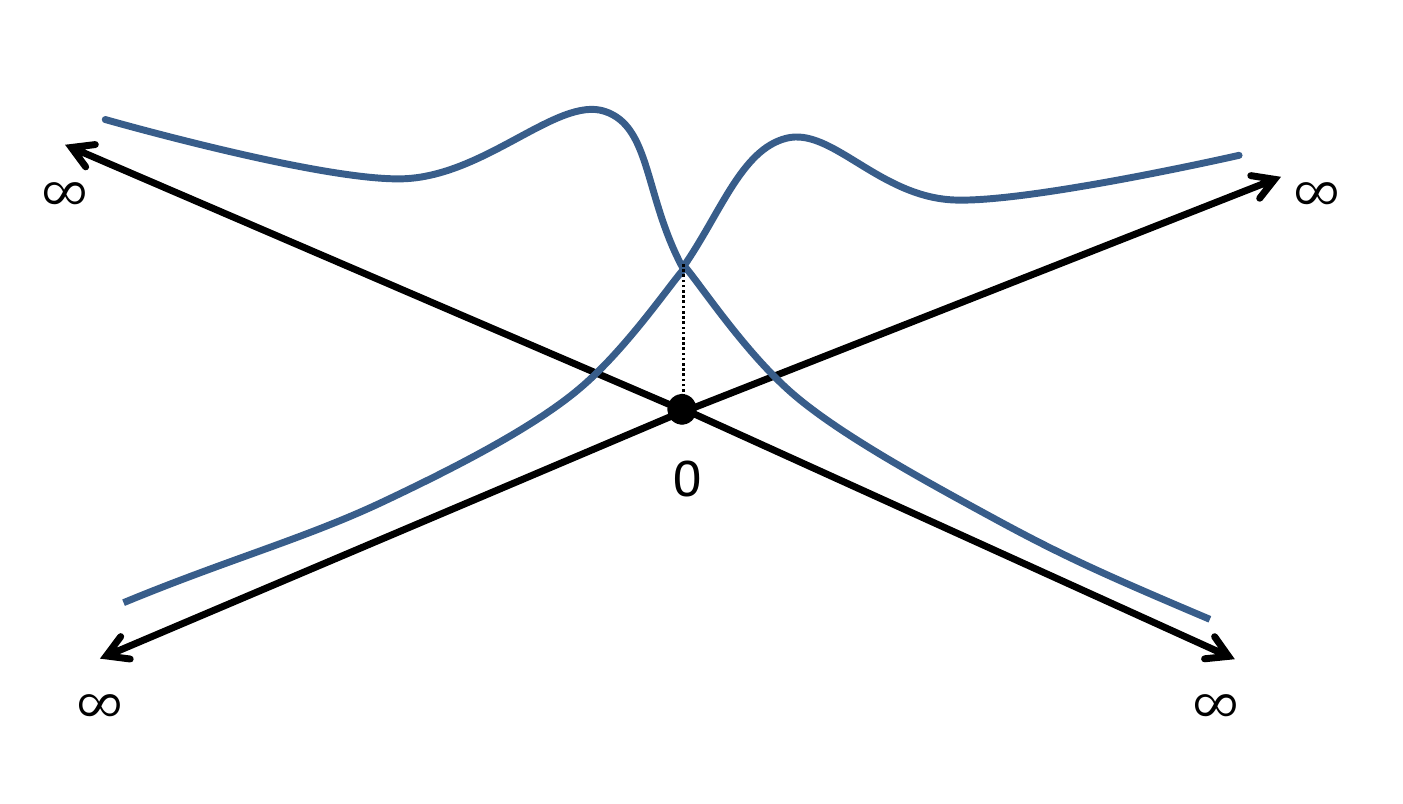}
\includegraphics[height=5.5cm,width=7.cm]{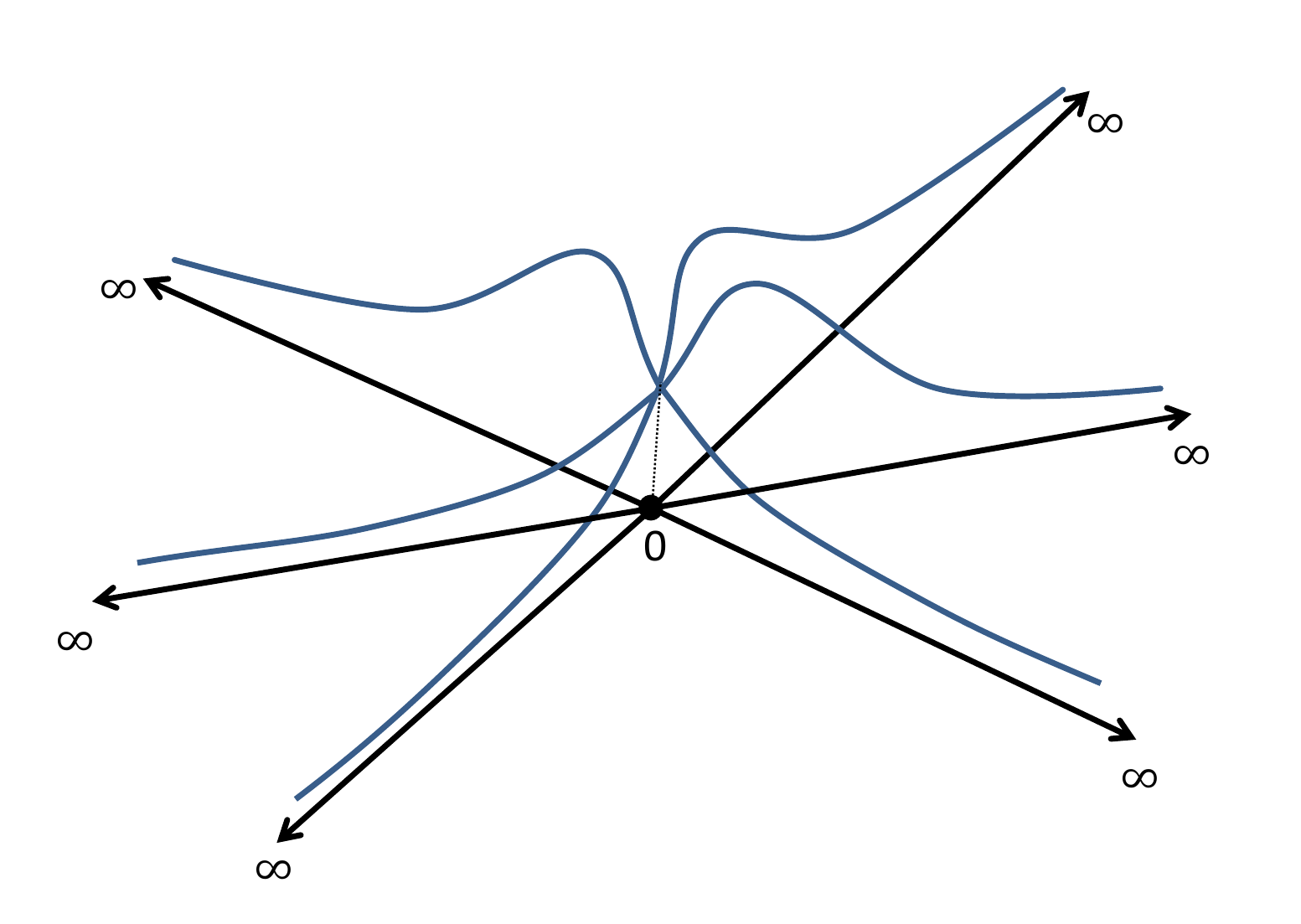}
\end{center}
\caption{Schematic representation of the shifted states (\ref{soliton-shifted}) with $a \neq 0$
for $N = 4$ (left) and $N = 6$ (right). The profile is monotonic in $N/2$ edges and
non-monotonic in the other $N/2$ edges.}
\label{fig-1}
\end{figure}

\begin{figure}[htb]
\begin{center}
\includegraphics[height=5.5cm,width=8.cm]{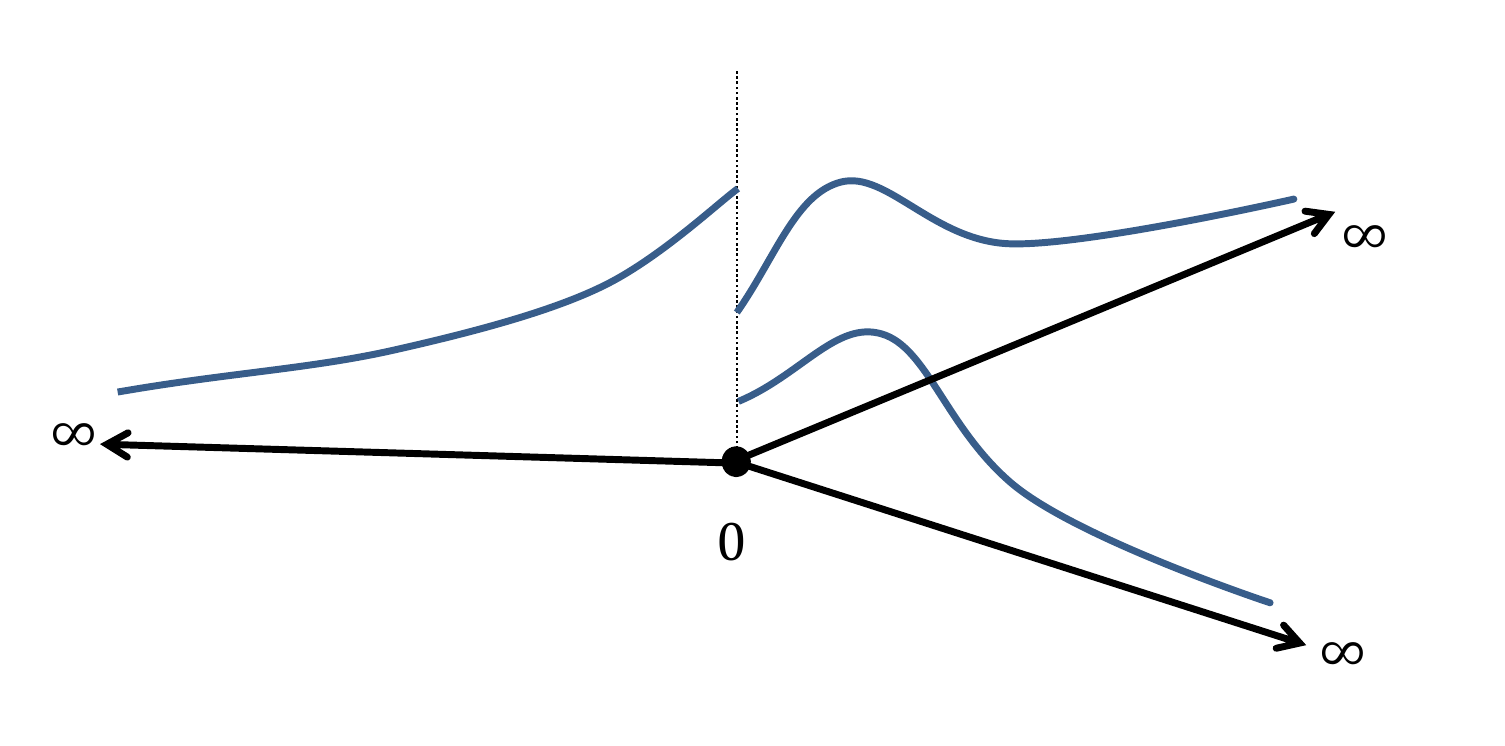}
\includegraphics[height=5.5cm,width=8.cm]{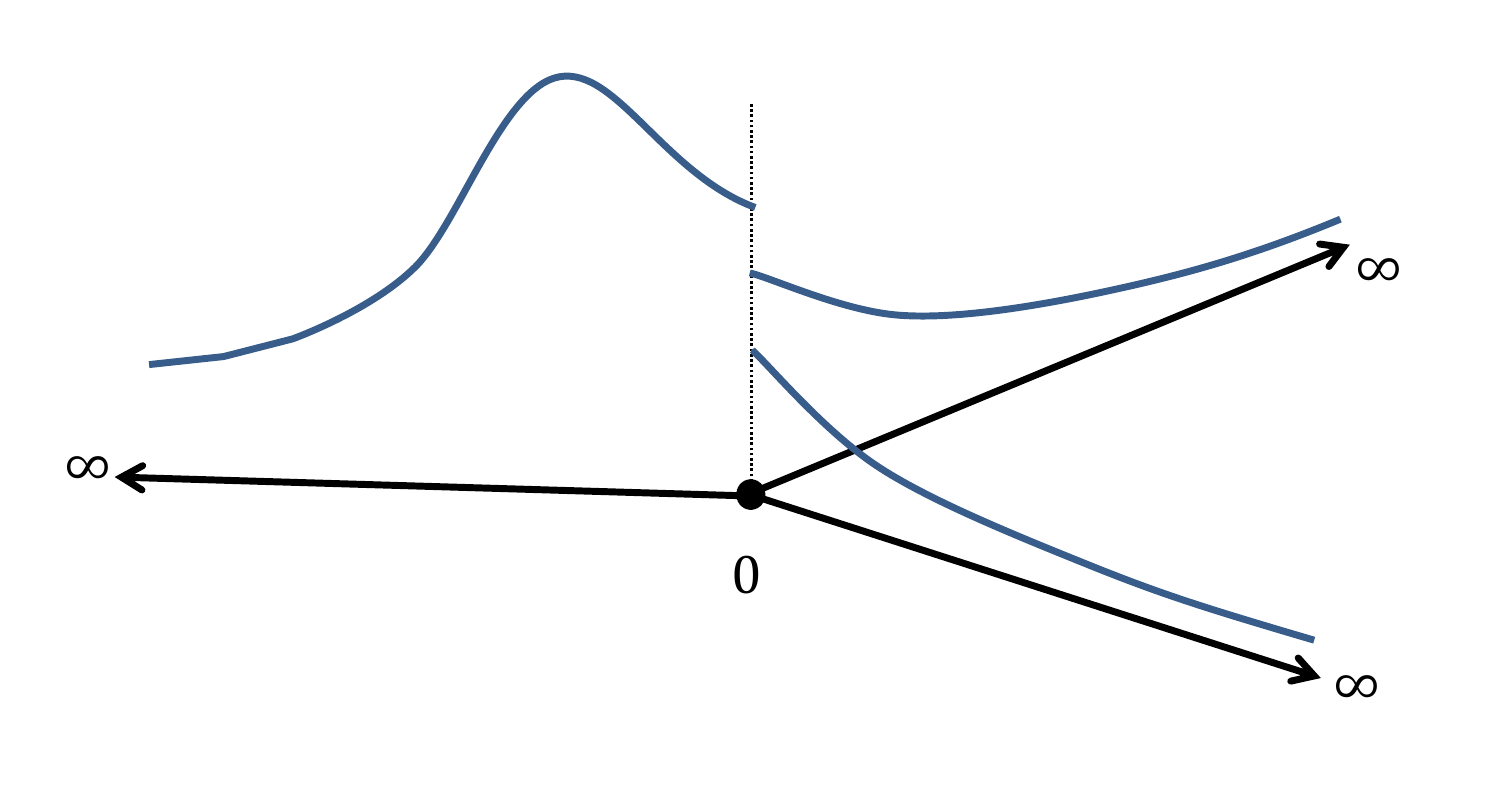}
\end{center}
\caption{Schematic representation of the shifted states (\ref{soliton-shifted-II}) with
$K = 1$, $N = 3$, and either $a > 0$ (left) or $a < 0$ (right). For $a > 0$,
the profile is  monotonic in $1$ edge and
non-monotonic in the other $2$ edges, whereas for $a < 0$, the situation is reversed.}
\label{fig-2}
\end{figure}

For graphical illustrations, we present some of the shifted states on Figure \ref{fig-1} and \ref{fig-2}.
Figure \ref{fig-1} shows shifted states in the case (i) with $N = 4$ (left) and $N = 6$ (right)
when the shifted states are given by Lemma \ref{solutions}.
If $a = 0$, the profile of $\Phi$ contains $N/2$ monotonic and $N/2$ non-monotonic tails
in different edges of the star graph $\Gamma$. Figures \ref{fig-2} shows shifted states in the case (ii)
with $N = 3$ when the shifted states are given by Lemma \ref{solutions-II} with $K = 1$.
If $a > 0$ (left), the profile of $\Phi$ contains $1$ monotonic and $2$ non-monotonic tails
whereas if $a < 0$ (right), the profile of $\Phi$ contains $2$ monotonic and $1$ non-monotonic tails.

\section{Main results on stability of the shifted states}

Every stationary state $\Phi$ satisfying the stationary NLS equation (\ref{eq3}) is a critical point of
the action functional
\begin{equation}
\label{action}
\Lambda(\Psi) := E(\Psi) + Q(\Psi), \quad \Psi \in H^1_{\Gamma},
\end{equation}
where $Q$ and $E$ are conserved mass and energy in (\ref{mass}) and (\ref{energy}) under the NLS flow,
by Lemma \ref{lemma-wellposedness}. Substituting $\Psi = \Phi + U + iW$
with real-valued $U,W \in H^1_{\Gamma}$ into $\Lambda(\Psi)$ and expanding in $U,W$ yield
\begin{equation}
\label{second-variation-Lambda}
\Lambda(\Phi + U + i W) = \Lambda(\Phi) + \langle L_+ U, U \rangle_{L^2(\Gamma)} +
\langle L_- W, W \rangle_{L^2(\Gamma)} + N(U,W),
\end{equation}
where
\begin{eqnarray*}
\langle L_+ U, U \rangle_{L^2(\Gamma)} & := & \int_{\Gamma}
\left[ (\nabla U)^2 + U^2 - (2p+1) (p+1) \alpha^2 \Phi^{2p} U^2 \right] dx, \\
\langle L_- W, W \rangle_{L^2(\Gamma)} & := & \int_{\Gamma}
\left[ (\nabla W)^2 + W^2 - (p+1) \alpha^2 \Phi^{2p} W^2 \right] dx,
\end{eqnarray*}
and $N(U,W) = {\rm o}(\| U \|_{H^1(\Gamma)}^2 + \| W \|_{H^1(\Gamma)}^2)$ for every $p > 0$.
The quadratic forms are defined by the two Hessian operators
\begin{eqnarray}
\label{Lplus}
L_+ & = & -\Delta + 1 - (2p+1) (p+1) \alpha^2 \Phi^{2p} : H^2_{\Gamma} \subset L^2(\Gamma) \to L^2(\Gamma), \\
L_- & = & -\Delta + 1 - (p+1) \alpha^2 \Phi^{2p} : \phantom{texttext} H^2_{\Gamma} \subset L^2(\Gamma) \to L^2(\Gamma),
\label{Lminus}
\end{eqnarray}
The number of negative eigenvalues of $L_+$ and $L_-$ is referred to as the Morse index of the shifted state $\Phi$.
The following theorem represents the main result of this paper.

\begin{theorem}
\label{theorem-eigenvalues}
Let $\Phi$ be a shifted state given by Lemma \ref{solutions-II} with $a \neq 0$.
Then $\sigma_p(L_-) \geq 0$ and $0$ is a simple eigenvalue of $L_-$,
whereas the non-positive part of $\sigma_p(L_+)$ consists of a simple eigenvalue $\lambda_0 < 0$,
another eigenvalue $\lambda_1 \in (\lambda_0,0)$ of multiplicity $K-1$ for $a < 0$ and $N-K-1$ for $a > 0$,
and a simple zero eigenvalue. The rest of $\sigma_p(L_-)$ and
$\sigma_p(L_+)$ is strictly positive and is bounded away from zero.
\end{theorem}

\begin{remark}
If $a = 0$, it was established in our previous work \cite{KP} that
the non-positive part of $\sigma_p(L_+)$ for the half-solitons (\ref{half-soliton}) consists
of a simple eigenvalue $\lambda_0 < 0$ and a zero eigenvalue of multiplicity $N-1$.
\end{remark}

By using the well-known results for the NLS equation
\cite{Grillakis} (see also \cite{CPV}), we can deduce spectral instability of
the shifted states from Theorem \ref{theorem-eigenvalues}.

\begin{corollary}
\label{corollary-instability}
With the exception of case (i) with $N = 2$ or case (ii) with $a < 0$,
the shifted states with $a \neq 0$
are spectrally unstable in the time evolution of the NLS equation (\ref{eq1}),
in particular, there exists real positive eigenvalues $\lambda$ in the spectral stability problem
\begin{equation}
\lambda \left[ \begin{array}{c} U \\ W \end{array} \right] =
\left[ \begin{array}{cc} 0 & L_- \\ -L_+ & 0 \end{array} \right] \left[ \begin{array}{c} U \\ W \end{array} \right].
\end{equation}
Moreover, with the account of algebraic multiplicity, for $p \in (0,2)$, there exist
$K-1$ real positive eigenvalues $\lambda$ for $a < 0$ and $N-K-1$ real positive eigenvalues for $a > 0$.
The shifted states of case (i) with $N = 2$ or case (ii) with $a < 0$ are spectrally stable for $p \in (0,2)$.
\end{corollary}

\begin{remark}
The result of Theorem \ref{theorem-eigenvalues} and Corollary \ref{corollary-instability} in case (i)
agrees with the qualitative picture described in Remark 5.3 in \cite{AdamiJDE1} and proves
the conjecture formulated in Remark 5.4 in \cite{AdamiJDE1} that all shifted states
(\ref{soliton-shifted}) given by
Lemma \ref{solutions} are unstable for all even $N \geq 4$.
If $N$ is even, the graph can be considered as a set of $K = N/2$ copies of the real line
and the shifted state can be interpreted as $K = N/2$ identical solitary waves on each real line
translated by the shift parameter $a \in \mathbb{R}$.
Since $L_-$ is positive at each solitary wave with a simple zero eigenvalue
and $L_+$ has a simple negative and a simple zero eigenvalue at each solitary wave,
we can count $N/2$ negative eigenvalues (with the account of their multiplicity),
in agreement with the statement of Theorem \ref{theorem-eigenvalues}.
However, the multiplicity of the zero eigenvalue is not explained in this qualitative picture,
and the count is incorrect for the half-soliton state (\ref{half-soliton}), which corresponds to the case $a = 0$,
see Remark \ref{remark-a-0}.
\end{remark}

\begin{remark}
Nonlinear instability of the half-soliton state (\ref{half-soliton}) in case (i) with even $N \geq 4$
is proved in the previous paper \cite{KP} based on the characterization
of the half-soliton state (\ref{half-soliton}) as a degenerate saddle point of the
action functional (\ref{action}) under the constraint of fixed mass $Q$.
The latter variational characterization was also given in \cite{AdamiJPA}.
The same argument holds for the half-soliton state (\ref{half-soliton}) for any $N \geq 3$.
\end{remark}

\begin{remark}
The difference between the normal form equations in \cite{KP} for odd and even $N$
clearly indicates the presence of shifted states in case (i). The zero equilibrium state
of the normal form equations is isolated for odd $N$, whereas it occurs at the intersection of $C_N$ lines
of equilibria for even $N$, where $C_N$ is given by (\ref{C-N}). In the latter case, for $N \geq 4$,
transverse directions in the energy space to the $C_N$ lines of equilibria imply
the saddle point geometry of the half-soliton state (\ref{half-soliton}) as is proven in \cite{KP}
for every $N \geq 3$.
\end{remark}

\begin{remark}
For the two spectrally stable shifted states in Corollary \ref{corollary-instability},
the case (i) with $N = 2$ is orbitally stable since the NLS equation on the star graph with $N = 2$
is completely equivalent to the NLS equation on the infinite line $\mathbb{R}$.
The case (ii) with $a < 0$ is more challenging since the orbit of the shifted state has
to be two-parametric due to phase rotation and translation in space, whereas the
star graph $\Gamma$ is not equivalent to the real line $\mathbb{R}$ if $N \geq 3$.
The translation in space can not be simply incorporated in the consideration because of
the vertex point of the star graph. In Section 7, we give
a simple argument why the shifted states (\ref{soliton-shifted-II})
with $a < 0$ are expected to be nonlinearly unstable under the NLS flow.
\end{remark}

\section{Count of the Morse index for the shifted states}

In order to prove Theorem \ref{theorem-eigenvalues}, we observe that the operators
$L_+$ and $L_-$ defined in (\ref{Lplus}) and (\ref{Lminus}) are self-adjoint in $L^2(\Gamma)$.
Since the bounded and exponentially decaying potential $\alpha^2 \Phi^{2p}$ is a relatively compact perturbation to
the unbounded operator $L_0 := -\Delta+1$, the absolutely continuous spectra of $L_\pm$,
by Weyl's Theorem, is given by $\sigma_c(L_{\pm}) = \sigma(L_0) = [1,\infty)$.
Therefore, we are only concerned about the eigenvalues of $\sigma_p(L_{\pm})$ in $(-\infty,1)$.

Since $\Phi(x) > 0$ for all $x \in \Gamma$, the same arguments as in Lemma 3.1 in \cite{KP} imply
that $\sigma_p(L_-)$ is nonnegative, $0 \in \sigma_p(L_-)$ is a simple eigenvalue with the eigenvector
$\Phi$, and all other eigenvalues in $\sigma_p(L_-)$ are bounded away from zero.
Hence we only need to consider $\sigma_p(L_+)$ in $(-\infty,1)$.

Spectral analysis involving various extensions of the operator $L_+$,
Neumann formula, and the count of the deficiency index is used in \cite{Pava}
in a similar context. Instead of developing a functional-analytic technique, we will prove
the assertion of Theorem \ref{theorem-eigenvalues} about $\sigma_p(L_+)$ by using Sturm's
nodal count for the scalar Schr\"{o}dinger equations.

Sturm theory is well-known in the context of the Sturm--Liouville boundary-value problem
on the finite interval (see, e.g., Section 5.1 in \cite{Teschl}).
This theory was generalized for bounded graphs in a number of recent publications
(see review in Section 5.2 in \cite{Kuchment}). Extension of this theory
to the unbounded graphs, e.g. to the star graph $\Gamma$, is not known
to the best of author's knowledge. As an outcome of our work, we will show that
the Sturm's nodal count is valid for the shifted states with $a \neq 0$ on the star graph $\Gamma$
in the same way as it is valid to the case of finite intervals or bounded graphs.

By using the representation (\ref{soliton-shifted-II}), let us consider
the exponentially decaying solutions of the second-order differential equation
\begin{equation} \label{shiftedeqofcomp}
- u''(x) + u(x) - (2p+1)(p+1) \sech^2(p(x + a)) u(x) = \lambda u(x), \quad x \in (0,\infty), \quad \lambda < 1,
\end{equation}
where $a \in \mathbb{R}$ is a parameter. By means of the substitution $u(x) = v(x + a)$ for $x \in (0,\infty)$,
exponentially decaying solutions $u$ to the equation (\ref{shiftedeqofcomp}) are equivalent to
exponentially decaying solutions $v$ of the second-order differential equation
\begin{equation} \label{eqofcomp}
- v''(x) + v(x) - (2p+1)(p+1) \sech^2(px) v(x) = \lambda v(x), \quad x \in (a,\infty), \quad \lambda < 1.
\end{equation}

The following lemmas extend some well-known results on the scalar Schr\"{o}dinger equation (\ref{eqofcomp}).

\begin{lemma}
\label{solhyp}
For every $\lambda<1$, there exists a unique solution $v \in C^1(\mathbb{R})$ to equation (\ref{eqofcomp})
such that
\begin{equation}
\label{u-limit}
\lim_{x \to +\infty} v(x) e^{\sqrt{1-\lambda} x} = 1.
\end{equation}
Moreover, for any fixed $x_0 \in \mathbb{R}$, $v(x_0)$ is a $C^1$ function of $\lambda$ for $\lambda < 1$.
The other linearly independent solution to equation (\ref{eqofcomp}) diverges as $x \to +\infty$.
\end{lemma}

\begin{proof}
The proof is based on the reformulation of the boundary--value problem (\ref{eqofcomp})--(\ref{u-limit})
as Volterra's integral equation. By means of Green's function, the solution to (\ref{eqofcomp})--(\ref{u-limit})
can be found from the inhomogeneous integral equation
\begin{equation}
\label{Volterra}
v(x) = e^{-\sqrt{1-\lambda}x} - \frac{(2p+1)(p+1)}{\sqrt{1-\lambda}}
\int_x^{\infty} \sinh(\sqrt{1-\lambda}(x-y)) \sech^2(py) v(y) \diff y.
\end{equation}
Setting $w(x;\lambda) = v(x) e^{\sqrt{1-\lambda}x}$ yields the
following Volterra's integral equation with a bounded kernel:
\begin{equation}
\label{VolterraMod}
w(x;\lambda) = 1 + \frac{(2p+1)(p+1)}{2\sqrt{1-\lambda}} \int_x^{\infty}  ( 1-e^{-2\sqrt{1-\lambda}(y-x)}) \sech^2(py) w(y;\lambda) \diff y.
\end{equation}
By standard Neumann series, the existence and uniqueness of a bounded solution $w(\cdot;\lambda) \in C^1(x_0,\infty)$
with $\lim_{x \to \infty} w(x;\lambda) = 1$ is obtained for every $\lambda < 1$ and sufficiently large $x_0 \gg 1$.
By the ODE theory, this solution is extended globally as a solution $w(\cdot;\lambda) \in C^1(\mathbb{R})$
of the integral equation (\ref{VolterraMod}). This construction yields a solution $v \in C^1(\mathbb{R})$
to the differential equation (\ref{eqofcomp}) with the exponential decay as $x \to +\infty$ given by (\ref{u-limit}).
Since the Volterra's integral equation (\ref{Volterra}) depends analytically on $\lambda$ for $\lambda < 1$,
then $v(x_0)$ is (at least) $C^1$ function of $\lambda < 1$ for any fixed $x_0 \in \mathbb{R}$.
Thanks to the $x$-independent and nonzero Wronskian determinant between two linearly independent solutions
to the second-order equation (\ref{eqofcomp}), the other linearly independent solution diverges
exponentially as $x \to +\infty$.
\end{proof}

\begin{lemma}
\label{lem-Sturm}
Let $v$ be the solution defined in Lemma \ref{solhyp}. If $v(0) = 0$ (resp. $v'(0) = 0$) for some $\lambda_0 < 1$,
then the corresponding eigenfunction $v$ to the Schr\"{o}dinger equation (\ref{eqofcomp}) is
an odd (resp. even) function on $\mathbb{R}$, whereas $\lambda_0$ is an eigenvalue of
the associated Schr\"{o}dinger operator defined in $L^2(\mathbb{R})$.
There exists exactly one $\lambda_0 < 0$ corresponding to $v'(0) = 0$ and a simple eigenvalue $\lambda_0 = 0$
corresponding to $v(0) = 0$, all other possible points $\lambda_0$ are located in $(0,1)$ bounded away from zero.
\end{lemma}

\begin{proof}
Extension of $v$ to an eigenfunction of the associated Schr\"{o}dinger operator defined in $L^2(\mathbb{R})$
follows by the reversibility of the Schr\"{o}dinger equation (\ref{eqofcomp}) with
respect to the transformation $x \mapsto -x$. The count of eigenvalues follows by Sturm's Theorem since
the odd eigenfunction for the eigenvalue $\lambda_0 = 0$,
\begin{equation}
\label{derivative-mode}
\phi'(x) = - \sech^{\frac{1}{p}}(px) \tanh(px)
\end{equation}
has one zero on the infinite line. Hence, $\lambda_0 = 0$ is the second eigenvalue of
the Schr\"{o}dinger equation (\ref{eqofcomp}) with exactly one simple negative eigenvalue $\lambda_0 < 0$
that corresponds to an even eigenfunction.
\end{proof}

\begin{lemma}
\label{Sturm}
Let $v = v(x;\lambda)$ be the solution defined by Lemma \ref{solhyp}. Assume that $v(x;\lambda_1)$
has a simple zero at $x = x_1 \in \mathbb{R}$ for some $\lambda_1 \in (-\infty,1)$.
Then, there exists a unique $C^1$ function
$\lambda \mapsto x_0(\lambda)$ for $\lambda$ near $\lambda_1$ such that $v(x;\lambda)$ has a simple
zero at $x = x_0(\lambda)$ with $x_0(\lambda_1) = x_1$ and $x_0'(\lambda_1) > 0$.
\end{lemma}

\begin{proof}
By Lemma \ref{solhyp}, $v$ is a $C^1$ function of $x$ and $\lambda$ for every $x \in \mathbb{R}$
and $\lambda \in (-\infty,1)$. Since $x_1$ is a simple zero of $v(x;\lambda_1)$,
we have $\partial_x v(x_1;\lambda_1) \neq 0$. By the implicit function theorem,
there exists a unique $C^1$ function $\lambda \mapsto x_0(\lambda)$ for $\lambda$ near $\lambda_1$
such that $v(x;\lambda)$ has a simple
zero at $x = x_0(\lambda)$ with $x_0(\lambda_1) = x_1$.
It remains to show that $x_0'(\lambda_1) > 0$.

Differentiating $v(x_0(\lambda);\lambda) = 0$ in $\lambda$ at $\lambda = \lambda_1$,
we obtain
\begin{equation}
\label{root-der-eq}
\partial_x v(x_1;\lambda_1) x_0'(\lambda_1) + \partial_{\lambda} v(x_1;\lambda_1) = 0.
\end{equation}
Let us denote $\tilde{v}(x) = \partial_{\lambda} v(x;\lambda_1)$. Differentiating
equation (\ref{eqofcomp}) in $\lambda$ yields the inhomogeneous
differential equation for $\tilde{v}$:
\begin{equation} \label{der-v-eq}
- \tilde{v}''(x) + \tilde{v}(x) - (2p+1)(p+1) \sech^2(px) \tilde{v}(x) =
\lambda_1 \tilde{v}(x) + v(x;\lambda_1), \quad x \in (a,\infty), \quad \lambda < 1.
\end{equation}
By the same method based on the Volterra's integral equation as in Lemma \ref{solhyp},
the function $\tilde{v}$ is $C^1$ in $x$ and decays to zero as $x \to \infty$.
Therefore, by multiplying equation (\ref{der-v-eq}) by $v(x;\lambda_1)$,
integrating by parts on $[x_1,\infty)$,
and using equation (\ref{eqofcomp}), we obtain
\begin{equation} \label{der-vv-eq}
- \partial_x v(x_1;\lambda_1) \tilde{v}(x_1) = \int_{x_1}^{\infty} v(x;\lambda_1)^2 dx,
\end{equation}
where we have used $v(x_1;\lambda_1) = 0$ as well as the decay
of $v(x;\lambda_1)$, $\partial_x v(x;\lambda_1)$, $\tilde{v}(x)$, and $\tilde{v}'(x)$
to zero as $x \to \infty$. Combining (\ref{root-der-eq}) and (\ref{der-vv-eq}) yields
\begin{equation}
\label{zeroshift}
(\partial_x v(x_1;\lambda_1) )^2 x_0'(\lambda_1) = \int_{x_1}^{\infty} v(x;\lambda_1)^2 dx > 0,
\end{equation}
so that $x_0'(\lambda_1) > 0$ follows from the fact that $\partial_x v(x_1;\lambda_1) \neq 0$.
\end{proof}

\begin{figure}[htb]
\begin{center}
\includegraphics[height=12.cm,width=14.cm]{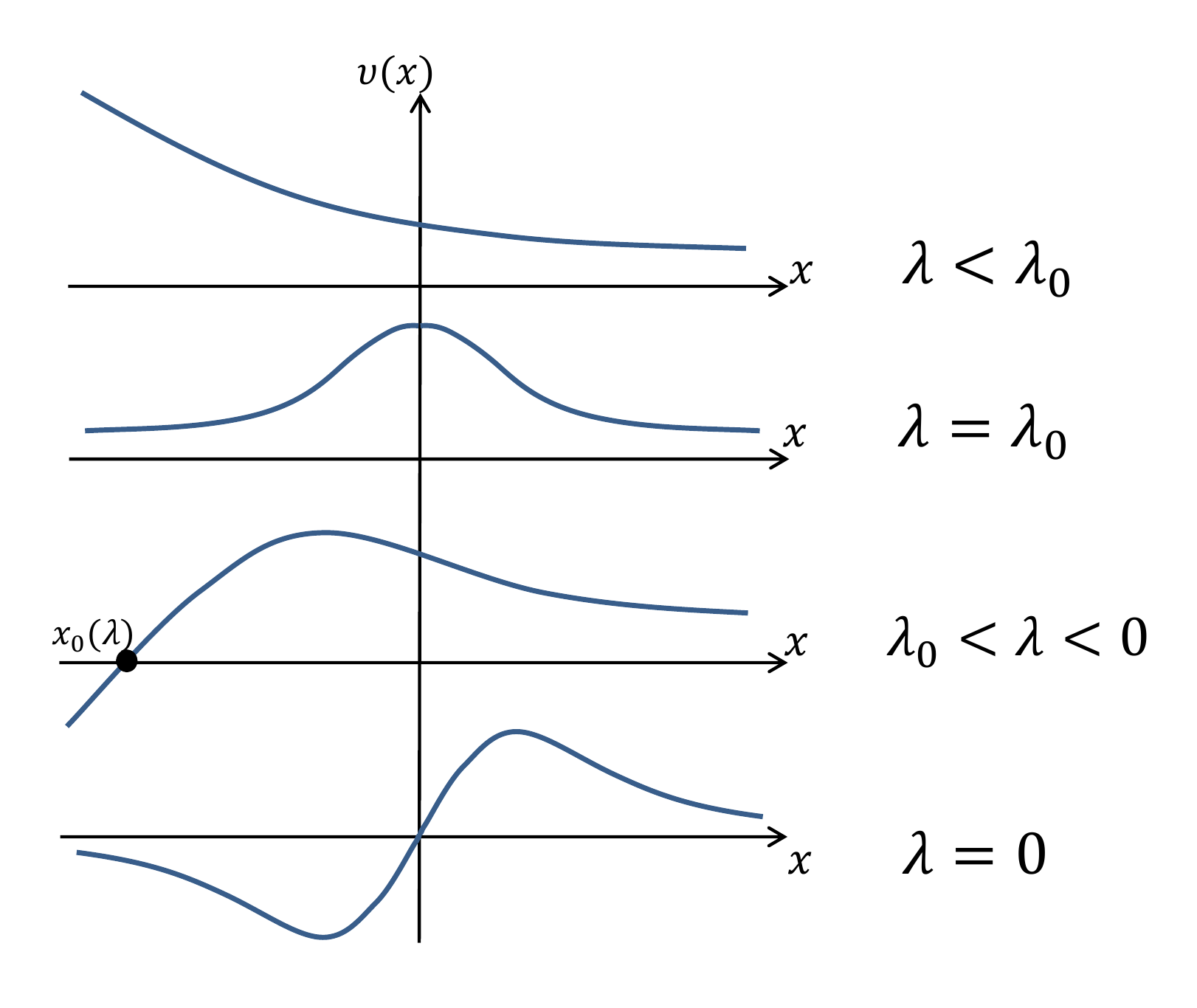}
\end{center}
\caption{Profiles of the solution $v$ of the differential equation (\ref{eqofcomp}) with
the limit (\ref{u-limit}).}
\label{fig-3}
\end{figure}

The results of Lemmas \ref{solhyp}, \ref{lem-Sturm}, and \ref{Sturm} are illustrated
on Figure \ref{fig-3} which shows profile of the solution $v$ satisfying the limit (\ref{u-limit})
for four cases of $\lambda$ in $(-\infty,0]$. The even eigenfunction for $\lambda_0 < 0$
and the odd eigenfunction for $\lambda = 0$ correspond to the solutions of the Schr\"{o}dinger
equation defined in $L^2(\mathbb{R})$. The only zero $x_0(\lambda)$ of $v$ appears from
negative infinity at $\lambda = \lambda_0$ and it is a monotonically increasing function
of $\lambda$ in $(\lambda_0,0)$ such that $x_0(0) = 0$.

By using the results of Lemmas \ref{solhyp}, \ref{lem-Sturm}, and \ref{Sturm}, we can address eigenvalues
of $\sigma_p(L_+)$ in $(-\infty,1)$, where the differential expression for $L_+$ is given by (\ref{Lplus}).

\begin{lemma}
\label{lem-eig}
Let $v$ be the solution defined by Lemma \ref{solhyp}. For every $a \in \mathbb{R}$,
$\lambda_0 \in (-\infty, 1)$ is an eigenvalue of $\sigma_p(L_+)$ if and only if one of the following equations holds: \\
(a) $v(a) = 0$,\\
(b) $v(-a) = 0$,\\
(c) $v(-a)v'(a)+v(a)v'(-a)=0$.\\
Moreover, $\lambda_0 \in \sigma_p(L_+)$ has multiplicity $K-1$ in the case (a), $N-K-1$
in the case (b), and is simple in the case (c). If $\lambda_0$ satisfies several cases, then its multiplicity
is the sum of the multiplicities in each case.
\end{lemma}

\begin{proof}
Let $U \in H^2_{\Gamma}$ be the eigenvector of the operator $L_+$ for the eigenvalue
$\lambda_0 \in \sigma_p(L_+)$. By Sobolev embedding of $H^2(\mathbb{R}^+)$ into
$C^1(\mathbb{R}^+)$, both $U(x)$ and $U'(x)$ decay to zero as $x \to +\infty$.
By using the representation (\ref{soliton-shifted-II})
and the transformation of (\ref{shiftedeqofcomp}) to (\ref{eqofcomp}),
we can write $U = (u_1,\dots,u_N)^T$ in the form
$$
u_j(x) = \begin{cases} c_j v(x+a), & j=1,\dots,K, \\
c_j v(x-a), & j=K+1,\dots,N,
\end{cases}
$$
where $(c_1,c_2,\dots,c_N)$ are coefficients and $v$ is the solution defined in Lemma \ref{solhyp}.
The boundary conditions for $U \in H^2_{\Gamma}$ in (\ref{H1}) and (\ref{H2})
imply the homogeneous linear system on the coefficients on $(c_1,c_2,\dots,c_N)$:
\begin{equation}
\label{lin-system-1}
c_1 \alpha_1^{1/p} v(a) = \dots = c_K \alpha_K^{1/p} v(a) =
c_{K+1} \alpha_{K+1}^{1/p} v(-a) = \dots = c_N \alpha_N^{1/p} v(-a)
\end{equation}
and
\begin{equation}
\label{lin-system-2}
\sum_{j=1}^{K} c_j \alpha_j^{-1/p} v'(a) + \sum_{j=K+1}^N c_j \alpha_j^{-1/p} v'(-a) = 0.
\end{equation}
The associated matrix is{\small$$
\bordermatrix{ & & & & & & & & & \cr
&\alpha_1^{1/p} v(a) & -\alpha_2^{1/p} v(a) & 0 & \dots & 0 & 0 & \dots & 0 & 0 \cr
&\alpha_1^{1/p} v(a) & 0 & -\alpha_3^{1/p} v(a) & \dots & 0 & 0 & \dots & 0 & 0 \cr
&\vdots & \vdots & \vdots & \ddots & \vdots & \vdots & \ddots & \vdots & \vdots \cr
&\alpha_1^{1/p} v(a) & 0 & 0 & \dots & -\alpha_{K}^{1/p} v(a) & 0 & \dots & 0 & 0 \cr
&\alpha_1^{1/p} v(a) & 0 & 0 & \dots & 0 & -\alpha_{K+1}^{1/p} v(-a) & \dots & 0 & 0 \cr
&\vdots & \vdots & \vdots & \ddots & \vdots & \vdots & \ddots & \vdots & \vdots \cr
&\alpha_1^{1/p} v(a) & 0 & 0 & \dots & 0 & 0 & \dots & 0 & -\alpha_{N}^{1/p} v(-a) \cr
& b_1 & b_2 & b_3 & \dots & b_K & b_{K+1} & \dots & b_{N-1} & b_N \cr
}$$}
where
$$
b_j = \left\{ \begin{array}{l} \alpha_j^{-1/p} v'(a), \quad 1 \leq j \leq K \\
\alpha_j^{-1/p} v'(-a), \quad K+1 \leq j \leq N. \end{array} \right.
$$

In order to calculate the determinant of the associate matrix, we perform elementary column operations and
obtain a lower triangular matrix. Let the associate matrix be of the form
$[A^{0}_1  A^{0}_2 \dots A^{0}_N]$, where $A_j^{0}$ represents
the $j$-th column of the matrix in the beginning of the algorithm.
Then, we perform the following elementary column operations:
$$
[A^{0}_1 \; A^{0}_2 \; A^{0}_3 \; \dots \; A^{0}_N] \longrightarrow
[A^{1}_1 \; A^{1}_2 \; A^{1}_3 \; \dots \; A^{1}_N] :=
[A^{0}_1 \; A^{0}_2 + \alpha^{-1/p}_1 \alpha^{1/p}_2 A^{0}_1 \; A^{0}_3 \; \dots \; A^{0}_N],
$$
then
$$
[A^{1}_1 \; A^{1}_2 \; A^{1}_3 \; \dots \; A^{1}_N] \longrightarrow
[A^{2}_1 \; A^{2}_2 \; A^{2}_3 \; \dots \; A^{2}_N] :=
[A^{1}_1 \; A^{1}_2 \; A^{1}_3 + \alpha^{-1/p}_2 \alpha^{1/p}_3 A^{1}_2 \; \dots \; A^{1}_N],
$$
and so on, until the $K$-th step. At the $K$-th step, we need to take into account
the change of $v(a)$ to $v(-a)$ in the $K+1$-th column, hence the $K$-th step involves
$$
A^{K-1}_{K+1} \longrightarrow A^{K}_{K+1} := A^{K-1}_{K+1} + \frac{\alpha^{1/p}_{K+1} v(-a)}{\alpha^{1/p}_K v(a)} A^{K-1}_K.
$$
At the $(K+1)$-th and subsequent steps, no further changes of $v(-a)$ occurs, so that
we apply the same rule as the one before the $K$-th step in all subsequent transformations.
Finally, after $(N-1)$ transformations, we obtain a lower triangular matrix in the form:
{\small$$
\bordermatrix{ & & & & & & & & & \cr
& \alpha_1^{1/p} v(a) & 0 & 0 & \dots & 0 & 0 & \dots & 0 & 0 \cr
& \alpha_1^{1/p} v(a) & \alpha_2^{1/p} v(a) & 0 & \dots & 0 & 0 & \dots & 0 & 0 \cr
& \alpha_1^{1/p} v(a) & \alpha_2^{1/p} v(a) & \alpha_3^{1/p} v(a) & \dots & 0 & 0 & \dots & 0 & 0 \cr
& \vdots & \vdots & \vdots & \ddots & \vdots & \vdots & \ddots & \vdots & \vdots \cr
& \alpha_1^{1/p} v(a) & \alpha_2^{1/p} v(a) & \alpha_3^{1/p} v(a) & \dots & \alpha_{K}^{1/p} v(a) & 0 & \dots & 0 & 0 \cr
& \alpha_1^{1/p} v(a) & \alpha_2^{1/p} v(a) & \alpha_3^{1/p} v(a) & \dots & \alpha_{K}^{1/p} v(a) & \alpha_{K+1}^{1/p} v(-a) & \dots & 0 & 0 \cr
& \vdots & \vdots & \vdots & \ddots & \vdots & \vdots & \ddots & \vdots & \vdots \cr
& \alpha_1^{1/p} v(a) & \alpha_2^{1/p} v(a) & \alpha_3^{1/p} v(a) & \dots & \alpha_{K}^{1/p} v(a) & \alpha_{K+1}^{1/p} v(-a) & \dots &
\alpha_{N-1}^{1/p} v(-a) & 0 \cr
& B_1 & B_2 & B_3 & \dots & B_K & B_{K+1} & \dots & B_{N-1} & B_N \cr
}
$$}
where $\{ B_j \}_{j = 1}^N$ are some numerical coefficients, in particular, $B_1 = \alpha_1^{-1/p} v'(a)$ and
$$
B_N = \frac{\alpha_N^{1/p}}{v(a)} \left[ \sum_{j=1}^K \alpha_j^{-2/p} v'(a) v(-a)
+ \sum_{j=K+1}^N \alpha_j^{-2/p} v'(-a) v(a)  \right].
$$
Under the constraint (\ref{alpha_const}), the determinant of the lower triangular matrix
is evaluated in the form:
$$
\Delta = \left( \prod_{j=1}^N \alpha_j^{1/p} \right)
\left( \sum_{j=1}^K \alpha_j^{-2/p} \right)
v(a)^{K-1} v(-a)^{N-K-1} \left[ v(-a)v'(a)+v(a)v'(-a) \right].
$$
Therefore, $U \neq 0$ is the eigenvector of $L_+$ for the eigenvalue $\lambda_0 \in (-\infty,1)$
if and only if $\Delta = 0$, or equivalently, if either $v(a) = 0$ or $v(-a) = 0$ or $v(-a)v'(a)+v(a)v'(-a)=0$.

In the case of $v(a) = 0$ and $v(-a) \neq 0$, it follows from the linear system (\ref{lin-system-1}) that
$c_j = 0$ for all $K+1 \leq j \leq N$ and $c_j \in \mathbb{R}$ are arbitrary
for all $1 \leq j \leq K$. The linear equation (\ref{lin-system-2})
implies that $\sum_{j=1}^{K} c_j \alpha_j^{-1/p} = 0$, since $v'(a) \neq 0$ when $v(a)=0$.
Therefore, the eigenvalue $\lambda_0$ has a multiplicity $K-1$.

Similarly, the eigenvalue $\lambda_0$ has a multiplicity $N-K-1$ if $v(a) \neq 0$ and $v(-a) = 0$.

In the case $v(-a)v'(a)+v(a)v'(-a) = 0$ but $v(a) \neq 0$ and $v(-a) \neq 0$, the linear system (\ref{lin-system-1})
implies that all coefficients are related to one coefficient. The linear equation (\ref{lin-system-2}) is then
satisfied due to the constraint (\ref{alpha_const}) and $\lambda_0$ is a simple eigenvalue.

If several cases are satisfied simultaneously, then it follows from the linear system (\ref{lin-system-1}) and (\ref{lin-system-2})
that multiplicity of $\lambda_0$ is equal to the sum of the multiplicities for each of the cases.
\end{proof}

\begin{proof1}{\em of Theorem \ref{theorem-eigenvalues}.}
The result on $\sigma_p(L_-)$ is proved from Lemma 3.1 in \cite{KP}.
The construction of $\sigma_p(L_+)$ follows from Lemmas \ref{lem-Sturm}, \ref{Sturm}, and \ref{lem-eig},
as well as the continuation arguments.

The condition (c) in Lemma \ref{lem-eig} is satisfied if the solution $v$ in Lemma \ref{solhyp}
is either odd or even function of $a$. For the simple eigenvalue $\lambda_0 < 0$ in Lemma \ref{lem-Sturm},
the eigenfunction is even and positive. Hence, $v(a) \neq 0$ and $v(-a) \neq 0$,
so that $\lambda_0$ is a simple eigenvalue in $\sigma_+(L_+)$
by the case (c) in Lemma \ref{lem-eig}. The corresponding eigenvector $U \in H^2_{\Gamma}$ is strictly positive
definite on $\Gamma$.

For the simple zero eigenvalue in Lemma \ref{lem-Sturm}, the eigenfunction (\ref{derivative-mode})
is odd and positive on $(-\infty,0)$.
Since $v(a) \neq 0$ and $v(-a) \neq 0$ if $a \neq 0$.
then $0$ is a simple eigenvalue in $\sigma_+(L_+)$ by the case (c) in Lemma \ref{lem-eig}. The corresponding eigenvector
$U \in H^2_{\Gamma}$ can be represented in the form:
\begin{equation}
\label{zero-eigenvector}
U(x) = \begin{cases} \alpha_j^{-1/p} \phi'(x+a), & j=1,\dots,K \\
- \alpha_{j}^{-1/p} \phi'(x-a), & j = K+1,\dots,N
\end{cases}.
\end{equation}
which represent the translation of the shifted state (\ref{soliton-shifted-II})
with respect to parameter $a$.

No other values of $\lambda_0$ exists in $(-\infty,\lambda_2)$ such that
the condition (c) in Lemma \ref{lem-eig} is satisfied, where $\lambda_2 > 0$
is either the positive eigenvalue of the scalar Schr\"{o}dinger equation (\ref{eqofcomp})
or the bottom of $\sigma_c(L_+)$ at $\lambda_2 = 1$.

If $a > 0$, then we claim that $v(a) > 0$ for every $\lambda \in (-\infty,0]$.
Indeed, by Lemma \ref{Sturm}, simple zeros of $v$ are monotonically increasing functions of $\lambda$,
whereas no multiple zeros of $v$ may exist for nonzero solutions of the second-order differential equations.
Since the only zero of $v$ bifurcates from $x = -\infty$ at $\lambda = \lambda_0 < 0$
and reaches $x = 0$ at $\lambda = 0$, $v(x)$ remains positive for every $x > 0$ for $\lambda \in (-\infty,0]$.
Hence the condition (a) in Lemma \ref{lem-eig} is not satisfied for every $\lambda \in (-\infty,0]$.

We now consider vanishing of $v(-a)$ for $a > 0$ for the condition (b) in Lemma \ref{lem-eig}.
By the same continuation argument from Lemma \ref{Sturm},
there exists exactly one $\lambda_1 \in (\lambda_0,0)$ such that $v(-a) = 0$ for any given $a > 0$.
Since $v'(-a) \neq 0$ and $v(a) \neq 0$, $\lambda_1$ is an eigenvalue of $\sigma_p(L_+)$
of multiplicity $N-K-1$.

For $a < 0$, the roles of cases (a) and (b) are swaped. The condition (b) is never satisfied,
while the condition (a) is satisfied for exactly one $\lambda_1 \in (\lambda_0,0)$,
which becomes an eigenvalue of $\sigma_p(L_+)$ of multiplicity $K-1$.
The assertion of Theorem \ref{theorem-eigenvalues} is proved.
\end{proof1}

\begin{remark}
For $p = 1$, the solution $v$ in Lemma \ref{solhyp} is available in the closed analytic form:
$$
v(x) = e^{-\sqrt{1-\lambda}x} \frac{3-\lambda + 3 \sqrt{1-\lambda} \tanh x - 3 \sech^2 x}{3 - \lambda + 3 \sqrt{1-\lambda}}.
$$
In this case, $\lambda_0 = -3$ is a simple eigenvalue corresponding to $v(x) = \frac{1}{4} \sech^2 x$
and $0$ is a simple eigenvalue corresponding to $v(x) = \frac{1}{2} \tanh x \sech x$.
If $a \neq 0$, the negative eigenvalue $\lambda_1 \in (\lambda_0,0)$ in
the proof of Theorem \ref{theorem-eigenvalues} is given by the root of the following transcendental equation
$$
3 - \lambda - 3 \sqrt{1-\lambda} \tanh|a| - 3 {\rm sech}^2(a) = 0,
$$
or explicitly, by
$$
\lambda_1 = -\frac{3}{2} \tanh|a| \left[ \tanh|a| + \sqrt{1+3 \sech^2(a)} \right],
$$
We note that $\lambda_1 \to 0$ when $a \to 0$ and $\lambda_1 \to \lambda_0 = -3$ when $|a| \to \infty$.
\end{remark}

\begin{remark}
By the count of Theorem \ref{theorem-eigenvalues}, the Morse index of $L_+$ is $K$ if $a < 0$
and $N - K$ if $a > 0$. On the other hand, the Sturm index (defined as the number of nodes
for the eigenfunction $U$ in (\ref{zero-eigenvector}) corresponding to the eigenvalue $\lambda = 0$)
is $K$ if $a < 0$ and $N-K$ if $a > 0$. Hence, the two indices are equal to each other,
similarly to the Sturm's nodal count for finite intervals or bounded graphs.
\end{remark}

\section{Homogenization of the star graph}

The translational symmetry in the star graph $\Gamma$ is broken due to the vertex at $x = 0$.
As a result, a momentum functional is not generally conserved under the NLS flow.
However, we will show here that if the coefficients $(\alpha_1,\alpha_2,\dots,\alpha_N)$ satisfy the
constraint (\ref{alpha_const}), then there exist solutions to the NLS equation (\ref{eq1}),
for which the following momentum functional is conserved:
\begin{equation}
\label{momentum}
P(\Psi) := \sum_{j=1}^N (-1)^{m_j} \int_{\mathbb{R}^+} {\rm Im} \left( \psi'_j \overline{\psi}_j \right) dx,
\end{equation}
where the $N$-tuple $(m_1, m_2, \dots, m_N)$ is given by (\ref{m_j}).
Computations yields the following momentum balance equation.

\begin{lemma}
\label{lemma-momentum}
For every $p > 0$ and every $(\alpha_1,\alpha_2,\dots,\alpha_N)$ satisfying the constraint (\ref{alpha_const}),
the local solution (\ref{solution-in-H1}) in Lemma \ref{lemma-wellposedness} satisfies the momentum balance equation
\begin{equation}
\label{dPdt}
\frac{d}{dt} P(\Psi) = \sum_{j=1}^N (-1)^{m_j} |\psi_j'(0)|^2.
\end{equation}
for all $t \in (-t_0,t_0)$, where $P$ is given by (\ref{momentum}).
\end{lemma}

\begin{proof}
If $p \geq 1$, we can consider the smooth solutions (\ref{solution-in-H3}) to the NLS equation (\ref{eq1}).
The momentum balance equation for $P$ in (\ref{momentum}) can be written in the form:
\begin{equation}
\label{momentum_t}
\frac{d}{dt} P(\Psi)  =  \sum_{j=1}^N (-1)^{m_j} \int_{\mathbb{R}^+}
{\rm Im} \left( \psi'_j \partial_t \overline{\psi}_j + \overline{\psi}_j \partial_t \psi'_j \right) dx.
\end{equation}
By using the formulas for $\partial_t \psi_j$ and $\partial_t \psi'_j$ obtained from the NLS equation (\ref{eq1}),
we can simplify the momentum balance equation (\ref{momentum_t}) and verify the following chain of equations
with the integration by parts technique:
\begin{eqnarray*}
\frac{d}{dt} P(\Psi) & = & \sum_{j=1}^N (-1)^{m_j} \int_{\mathbb{R}^+} {\rm Re} \left( \overline{\psi}_j \psi'''_j
- \psi'_j \overline{\psi}''_j + p\alpha^2_j (|\psi_j|^{2p+2})' \right) dx \\
& = & \sum_{j=1}^N (-1)^{m_j} \left( - {\rm Re}[\overline{\psi}_j(0) \psi''_j(0)]
+ |\psi'_j(0)|^2 - p \alpha^2_j |\psi_j(0)|^{2p+2} \right),
\end{eqnarray*}
where the decay of $\Psi(x)$, $\Psi'(x)$, and $\Psi''(x)$ to zero at infinity has been used
for the solution in $H^3_{\Gamma}$. Applying the boundary conditions in (\ref{H1}) and (\ref{H3}),
the constraint (\ref{alpha_const}), and the choice of values of $m_j$ in (\ref{m_j})
yields the momentum balance equation in the form (\ref{dPdt}).

Although our derivation was restricted to the case $p \geq 1$ and to solutions in $H^3_{\Gamma}$,
the proof can be extended to the local solution (\ref{solution-in-H1}) for
all values of $p > 0$ by using standard approximation techniques \cite{Caz}.
\end{proof}

The momentum $P(\Psi)$ is conserved in $t$ if the boundary
conditions for derivatives satisfy the additional constraints:
\begin{equation}
\label{assumption}
(-1)^{m_1} \alpha_1^{1/p} \psi_1'(0) = (-1)^{m_2} \alpha_2^{1/p} \psi_2'(0) = \dots = (-1)^{m_N} \alpha_N^{1/p} \psi_N'(0),
\end{equation}
which are compatible with the boundary conditions in (\ref{H2}) under the constraint (\ref{alpha_const}).
Under the constraint (\ref{assumption}), we verify from equation (\ref{dPdt}) that
$$
\frac{d}{dt} P(\Psi) = (-1)^{m_1} \alpha_1^{2/p} |\psi_1'(0)|^2
\left( \sum_{j=1}^K \frac{1}{\alpha_j^{2/p}} - \sum_{j=K+1}^N \frac{1}{\alpha_j^{2/p}} \right) = 0,
$$
hence $P(\Psi)$ is conserved in $t$.

In order to make sure that the constraint (\ref{assumption}) is satisfied for every $t$,
we observe the following reduction of the NSL equation (\ref{eq1}) on the star
graph $\Gamma$ to the homogeneous NLS equation on the infinite line.

\begin{lemma}
Under the constraint (\ref{alpha_const}), there exist solutions
of the NLS equation (\ref{eq1}) on the graph $\Gamma$ which satisfy the
the following homogeneous NLS equation on the infinite line:
\begin{equation}
\label{nls}
iU_t + U_{xx} + (p+1)\left| U \right|^{2p}U = 0, \quad x \in \mathbb{R}, \quad t \in \mathbb{R},
\end{equation}
where $U = U(t,x) \in \mathbb{C}$.
\end{lemma}

\begin{proof}
The class of suitable solutions $\Psi$ to the NLS equation (\ref{eq1}) on the star graph $\Gamma$
must satisfy the following reduction:
\begin{equation}
\label{Rsolutions}
\left\{ \begin{array}{l}
\alpha_1^{1/p} \psi_1(t,x) = \dots = \alpha_K^{1/p} \psi_K(t,x), \\
\alpha_{K+1}^{1/p} \psi_{K+1}(t,x) = \dots = \alpha_N^{1/p} \psi_N(t,x), \end{array} \right.
\quad x \in \mathbb{R}^+, \quad t \in \mathbb{R},
\end{equation}
subject to the boundary conditions at the vertex point $x = 0$:
\begin{equation}
\label{Rboundary}
\alpha_K^{1/p} \psi_K(t,0) = \alpha_{K+1}^{1/p} \psi_{K+1}(t,0), \quad
\alpha_K^{1/p} \partial_x \psi_K(t,0) = - \alpha_{K+1}^{1/p} \partial_x \psi_{K+1}(t,0).
\end{equation}
Note that the boundary conditions are compatible with the generalized Kirchhoff boundary conditions in
(\ref{H2}) under the constraint (\ref{alpha_const}). Thanks to the reduction (\ref{Rsolutions}),
the following function can be defined on the infinite line:
\begin{equation}
\label{Ufunction}
U(t,x) := \begin{cases}
\alpha_j^{1/p} \psi_j(t,-x), \quad 1 \leq j \leq K, \qquad x \in \mathbb{R}^-, \\
\alpha_j^{1/p} \psi_j(t,x), \quad K+1 \leq j \leq N, \quad x \in \mathbb{R}^+.
\end{cases}
\end{equation}
Thanks to the boundary conditions (\ref{Rboundary}), $U$ is a $C^1$ function across $x = 0$.
Substitution (\ref{Ufunction}) into the NLS equation (\ref{eq1}) on the graph $\Gamma$
yields the homogeneous NLS equation (\ref{nls}), where the point $x = 0$
is a regular point on the infinite line $\mathbb{R}$.
\end{proof}

\begin{remark}
The shifted state (\ref{soliton-shifted-II})
corresponds to the NLS soliton in the homogeneous NLS equation (\ref{nls}),
which is translational invariant along the line $\mathbb{R}$. The eigenvalue count
of Theorem \ref{theorem-eigenvalues} and the instability result of Corollary
\ref{corollary-instability} are related to the symmetry-breaking perturbations,
which do not satisfy the reduction (\ref{Rsolutions}). These perturbations satisfy
the NLS equation (\ref{eq1}) on the graph $\Gamma$ but do not satisfy
the homogeneous NLS equation (\ref{nls}) on the line $\mathbb{R}$. Such symmetry-breaking
perturbations were not considered in \cite{M1,M2,M3}.
\end{remark}

\section{Variational characterization of the shifted states}

Here we give a simple argument suggesting that the spectrally stable shifted states in the case (ii)
with $K = 1$ and $a < 0$ are nonlinearly unstable under the NLS flow. 
This involves the variational characterization of the
shifted states in the graph $\Gamma$ as critical points of energy under the fixed mass,
where the mass and energy are defined by (\ref{mass}) and (\ref{energy}) respectively.

The mass and energy are computed at the shifted states (\ref{soliton-shifted-II}) as follows:
\begin{equation}
\label{mass-shifted}
Q(\Phi) = \left( \sum_{j=1}^K \alpha_j^{-2/p} \right) \| \phi \|^2_{L^2(\mathbb{R})}
\end{equation}
and
\begin{equation}
\label{energy-shifted}
E(\Phi) = \left( \sum_{j=1}^K \alpha_j^{-2/p} \right) \left( \| \phi' \|^2_{L^2(\mathbb{R})} - \| \phi \|^{2p+2}_{L^{2p+2}(\mathbb{R})}\right),
\end{equation}
where the constraint (\ref{alpha_const}) has been used. In the case (ii) with $K = 1$,
the mass and energy at the shifted states is the same as the mass and energy of
a free solitary wave escaping to infinity along the incoming edge. This property
signals out that the infimum of energy is not achieved, as is discussed in \cite{AdamiJFA}.

Furthermore, the constraint (\ref{alpha_const}) implies that
$\alpha_2,\dots,\alpha_N > \alpha_1$ (if $N \geq 3$). Pick the $j$-th outgoing edge for $2 \leq j \leq N$
and fix the mass at the level $\mu > 0$. Then, it is well-known \cite{AdamiJFA}
that the energy of a free solitary wave escaping to infinity along the $j$-th outgoing
edge is given by
\begin{equation}
E_j = - C_p \alpha_j^{\frac{4}{2-p}} \mu^{\frac{p+2}{2-p}} < - C_p \alpha_1^{\frac{4}{2-p}} \mu^{\frac{p+2}{2-p}} = E(\Phi),
\end{equation}
where $p \in (0,2)$ and $C_p$ is an universal constant that only depends on $p$.
Thus, a free solitary wave escaping the graph $\Gamma$ along any outgoing edge has a
lower energy level at fixed mass compared to the shifted state. This suggests that any shifted state is
energetically unstable.

Let us now give a simple argument suggesting nonlinear instability of the shifted states (\ref{soliton-shifted-II})
for $a < 0$ under the NLS flow. If $K=1$, it follows from the momentum balance equation (\ref{dPdt})
in Lemma \ref{lemma-momentum} that the momentum $P(\Psi)$ defined by (\ref{momentum})
is an increasing function of time if $m_1 = 0$ and
a decreasing function of time if $m_1 = 1$, for the two choices in (\ref{m_j}).
Indeed, we obtain the following chain of transformations by using
the boundary conditions in (\ref{H2}) and the constraint (\ref{alpha_const}):
\begin{eqnarray}
\label{bal-mom}
\frac{d}{dt} P(\Psi) & = & (-1)^{m_1+1} \left[ \sum_{j=2}^N
|\psi_j'(0)|^2 - \sum_{j=2}^N \sum_{i=2}^N
\frac{\alpha_1^{2/p}}{\alpha_j^{1/p} \alpha_i^{1/p}}\psi_j'(0) \overline{\psi_i}'(0) \right] \\
\nonumber
& = & (-1)^{m_1+1} \left( \sum_{j=2}^N \sum_{\substack{i=2 \\ i \neq j}}^N \frac{\alpha_1^{2/p}}{\alpha_i^{2/p}} |\psi_j'(0)|^2
- \sum_{j=2}^N \sum_{\substack{i=2 \\ i \neq j}}^N \frac{\alpha_1^{2/p}}{\alpha_j^{1/p} \alpha_i^{1/p}}\psi_j'(0) \overline{\psi_i}'(0) \right)\\
\nonumber
& = & \frac{1}{2} (-1)^{m_1+1} \sum_{j=2}^N \sum_{\substack{i=2 \\ i \neq j}}^N \frac{\alpha_1^{2/p}}{\alpha_j^{2/p} \alpha_i^{2/p}}
\left| \alpha_j^{1/p} \psi_j'(0) - \alpha_i^{1/p} \psi_i'(0) \right|^2
\end{eqnarray}
Hence $\frac{d}{dt} P(\Psi) \geq 0$ if $m_1 = 1$ and $\frac{d}{dt} P(\Psi) \leq 0$ if $m_1 = 0$.

Since the shifted states (\ref{soliton-shifted-II}) satisfies $P(\Phi) = 0$, monotonicity of
the momentum $P(\Psi)$ in time $t$ immediately implies nonlinear instability of
the shifted states (\ref{soliton-shifted-II}) with $a < 0$ under the NLS flow,
despite that these shifted states are spectrally stable by Corollary \ref{corollary-instability}.
Indeed, if $a < 0$ (or $m_1 = 1$), the value of the momentum $P(\Psi)$ is monotonically increasing in time as soon as the
right-hand side of (\ref{bal-mom}) is nonzero. Therefore, if $P(\Psi)$ is initially near zero,
which is the value of $P(\Phi)$ for every shifted state (\ref{soliton-shifted-II}) with $a \in \mathbb{R}$,
then $P(\Psi)$ grows far away from the zero value. This simple argument leads us to the following conjecture.

\begin{conjecture}
\label{conjecture-instability}
In the case $K = 1$, the branch of shifted states (\ref{soliton-shifted-II}) with $a < 0$
is nonlinearly unstable under the NLS flow.
\end{conjecture}

We also add more details on the travelling waves of the homogeneous NLS equation on the infinite line
given by (\ref{nls}). Every stationary solution $U_{\rm stat}(x) e^{it}$ to the NLS equation (\ref{nls})
is translated into a family of the travelling solutions with speed $c$ by the Lorentz transformation
\begin{equation}
\label{U-trav}
U_{\rm trav}(t,x) = U_{\rm stat}(x-ct) e^{it + i cx/2 - i c^2 t/4}, \quad x \in \mathbb{R}, \quad t \in \mathbb{R}.
\end{equation}
Computing the momentum $P(\Psi)$ given by (\ref{momentum}) at the solution $\Psi$, which is defined by (\ref{Ufunction}) with
(\ref{U-trav}), yields
\begin{equation}
\label{momentum-U}
P(\Psi) = \frac{c}{2} \alpha_1^{2/p} (-1)^{m_1+1} \| U_{\rm stat} \|_{L^2(\mathbb{R})}^2.
\end{equation}
If $a < 0$ (or $m_1 = 1$), then $\frac{d}{dt} P(\Psi) \geq 0$ and the increase of the momentum
can be compensated by a travelling wave (\ref{U-trav})
moving with the speed $c > 0$ as follows from (\ref{momentum-U}). The maximum of the shifted state
(\ref{soliton-shifted-II}) with $a < 0$
is located in the only incoming edge but it may move towards the vertex point at $x = 0$ along the travelling wave
(\ref{U-trav}).

On the other hand, if $a > 0$ (or $m_1 = 0$), then $\frac{d}{dt} P(\Psi) \leq 0$ and the decrease of the momentum
can be compensated by the same travelling wave (\ref{U-trav}) still moving with the speed $c > 0$,
as follows from (\ref{momentum-U}).
Therefore, we may anticipate that the maximum of the shifted state (\ref{soliton-shifted-II})
first moves towards the vertex point at $x = 0$ along the only incoming edge,
then splits into $N-1$ maxima in the $N-1$ outgoing edges and these $N-1$ maxima keep moving outward the vertex point
at $x = 0$, subject to additional spectral instability due to symmetry-breaking perturbations 
between $N-1$ outgoing edges, according to Corollary \ref{corollary-instability}.
This dynamical picture is in agreement with the variational consideration above, but the proof
of validity of this dynamical picture is beyond the scope of this work and will be considered elsewhere.
Therefore, we summarize this dynamical picture as another conjecture.

\begin{conjecture}
In the case $K = 1$, the shifted state (\ref{soliton-shifted-II}) with $a < 0$ leads to a solitary wave that
moves towards the vertex point at $x = 0$ along the only  incoming edge,
splits into $N-1$ solitary waves in the $N-1$ outgoing edges, which transform due to their spectral instability
while moving outward the vertex point at $x = 0$.
\end{conjecture}

\end{document}